\documentclass[12pt,a4paper,oneside,reqno]{article}

\usepackage[utf8]{inputenc}
\usepackage[english]{babel}
\usepackage[left=1in,right=1in,top=1in,bottom=1in]{geometry}
\usepackage{amsfonts,amssymb,amsmath,amsthm}
\usepackage[symbol]{footmisc}
\usepackage{longtable,array}
\newcolumntype{L}[1]{>{\raggedright\arraybackslash$}p{#1}<{$}}
\usepackage{enumerate}
\usepackage{cite}
\usepackage{xcolor}

\usepackage{hyperref}
\hypersetup{hidelinks}
\usepackage{fancyhdr}
\theoremstyle{plain}
\newtheorem{theorem}{Theorem}
\newtheorem{lemma}[theorem]{Lemma}
\newtheorem{proposition}[theorem]{Proposition}
\newtheorem{corollary}[theorem]{Corollary}
\theoremstyle{definition}
\newtheorem{remark}[theorem]{Remark}
\newtheorem{definition}[theorem]{Definition}
\newtheorem{example}[theorem]{Example}
\makeatletter
\@removefromreset{theorem}{section}\@removefromreset{equation}{section}

\makeatother
\usepackage{adjustbox}
\usepackage{tikz}\usetikzlibrary{arrows}

\IfFileExists{fouriernc.sty}{\usepackage{fouriernc}}{}

\newcommand{\doi}[1]{\href{https://doi.org/#1}{doi:\detokenize{#1}}}
\newcommand{\arxiv}[1]{\href{https://arxiv.org/abs/#1}{arXiv:#1}}
\newcommand{\arxivsearch}[1]{\href{https://arxiv.org/search/?query=#1}{arXiv}}

\def\AP{\mathrm{AP}}

\def\Lie{\mathrm{Lie}}
\def\As{\mathrm{As}}
\def\Com{\mathrm{Com}}
\def\Var{\mathrm{Var}}
\def\MP{\mathrm{MP}}

\begin{document}

\noindent\fbox{\parbox{\dimexpr\textwidth-2\fboxsep-2\fboxrule\relax}{\small
\textbf{Corrected version (verified by Claude).} This is a corrected version of the paper H. Abdelwahab, I. Kaygorodov, B. Sartayev,
\emph{$\delta$-Poisson and transposed $\delta$-Poisson algebras}, Journal of Non-Associative Structures, 1 (2026), no. 1, 6, \doi{10.46298/jonas.18002}.}}

\bigskip

\begin{center}
{\Large $\delta$-Poisson and transposed $\delta$-Poisson algebras}

\bigskip

{\bf Hani Abdelwahab, Ivan Kaygorodov, and Bauyrzhan Sartayev}
\end{center}

\medskip

\noindent{\small\textbf{Abstract.} We study two new classes of Poisson-type algebras, namely, $\delta$-Poisson and transposed $\delta$-Poisson algebras.
We show that these algebras are related to many interesting identities; in particular, they are related to shift associative algebras, 
$F$-manifold algebras, algebras of Jordan brackets, etc.
We classify simple $\delta$-Poisson and transposed $\delta$-Poisson algebras and find their depolarizations.
We study the Koszulity and self-duality of the $\delta$-Poisson and mixed-Poisson operads.
Bases of the free $\delta$-Poisson and mixed-Poisson algebras generated by a countable set $X$ are constructed.}

\bigskip


\bigskip


\bigskip

\tableofcontents

\bigskip

\section*{Introduction}\addcontentsline{toc}{section}{Introduction}
 
Poisson algebras arose from the study of Poisson geometry in the 1970s and have appeared in an extremely wide range of areas in mathematics and physics, such as Poisson manifolds, algebraic geometry, operads, quantization theory, quantum groups, and classical and quantum mechanics. The study of Poisson algebras also led to other algebraic structures, such as 
non-commutative Poisson algebras, 
generic Poisson algebras,
transposed Poisson algebras,
Novikov--Poisson algebras,
algebras of Jordan brackets and generalized Poisson algebras,
$F$-manifold algebras,
quasi-Poisson algebras,
Gerstenhaber algebras,
Poisson bialgebras,
double Poisson algebras,
Poisson $n$-Lie algebras, etc.
$\big($see, \cite{fer23,LSB,aae23,dzhuma,AFR,LB,MF,TP1,kms,sh93,GR} and references in \cite{k23}$\big).$ 
All these classes are related in some way to Poisson algebras; for that reason, we call them Poisson-type algebras.

\medskip 

The present paper is devoted to the study of two new classes of Poisson-type algebras, namely, $\delta$-Poisson and transposed $\delta$-Poisson algebras.
The motivation for studying anti-Poisson (i.e., $\delta=-1$) algebras comes from a recent study of shift associative algebras \cite{aks24}; 
$\delta$-Poisson algebras arise as a generalization of Poisson and anti-Poisson algebras; on the other hand, they are closely related to $\delta$-derivations introduced by Filippov in \cite{fil1}.
The dual notion, obtained by exchanging the roles of the two binary operations in the Leibniz rule, is the notion of a transposed $\delta$-Poisson algebra (the notion of transposed Poisson algebras was recently introduced in \cite{TP1}).
We show that the behavior of $\delta$-Poisson algebras for $\delta\neq1$ differs from that of Poisson algebras. Namely, they are almost $3$-step nilpotent, and there are no nontrivial simple $\delta$-Poisson algebras (Proposition \ref{antixyzt} and Theorem \ref{dPsimple}). 
We prove that transposed $\delta$-Poisson algebras share many known identities from \cite{TP1}, but, unlike the case of transposed Poisson algebras, there are complex simple finite-dimensional transposed anti-Poisson algebras (Theorem \ref{simpleclass}).
The tensor product of two $\delta$-Poisson (resp. transposed $\delta$-Poisson) algebras gives a $\delta$-Poisson (resp. transposed $\delta$-Poisson) algebra (Theorems \ref{tensdp} and \ref{tenstdp}).
Section \ref{onemu} is devoted to the study of $\delta$-Poisson and transposed $\delta$-Poisson algebras in one multiplication (their depolarizations). For example, Proposition \ref{depol} gives identities of the depolarization of mixed-Poisson algebras (in particular, we have an identity of the first associative type $\sigma=(13)$ considered in \cite{aks24}).
Section \ref{free} describes bases of the free $\delta$-Poisson and mixed-Poisson algebras generated by a countable set.
Section \ref{dualoper} is devoted to the operad $\delta$-${\rm P}$ governed by the variety of $\delta$-Poisson algebras.
Namely, we prove that the operad $\delta$-${\rm P}$ is self-dual (Theorem \ref{26}) and that the operad governed by the variety of anti-Poisson algebras is not Koszul (Theorem \ref{27}).
 In Theorem \ref{36}, 
  we prove that the dual of the mixed-Poisson operad is isomorphic to the free product of the operads Lie and As-Com.

  \medskip 
  
  \noindent
{\bf Notations.}
We do not recall some well-known definitions
(such as the definitions of Lie algebras,
Jordan algebras, nilpotent algebras, solvable algebras,  idempotents, etc.) and refer the reader to previously published papers (for example, \cite{aks24,Albert, TP1}). 
For the commutator, anticommutator, and associator,  
we will use the standard notations:

\begin{center}
    $[x,y] : = \frac  12(xy-yx),$   \  
    $x \circ y : = \frac  12 (xy+yx)$ \ and \ 
    $(x,y,z):=(xy)z-x(yz).$
\end{center}
In general, we work over the complex field, but some results remain valid over other fields.
We also almost always assume that $\delta \neq 1$ and that the algebras with two multiplications under consideration are nontrivial, i.e., both multiplications are nonzero.
In commutative and anticommutative algebras, 
the full multiplication table can be recovered from the symmetry or antisymmetry of the multiplication.

\section{New Poisson-type algebras}
\subsection{$\delta$-Poisson algebras} 
\begin{definition}
An algebra $({\rm P}, \cdot, \{\cdot,\cdot\})$ is defined to be  an {anti-Poisson algebra} (resp., anti-Poisson-Jordan  algebra), 
if $({\rm P}, \cdot)$ is a  commutative associative  (resp., Jordan) algebra,
$({\rm P},  \{\cdot,\cdot\})$ is a Lie algebra and the following
identity holds$:$%
\begin{equation}\label{antipois}
\{x  y, z \} \  = \ -\big(x  \{ y,z \} +\{x,z\}  y\big).
\end{equation}

\end{definition}

\noindent
A motivation for studying anti-Poisson algebras comes from the study of shift associative algebras in \cite{aks24}.
Namely, the following observation was obtained there. 

\begin{theorem}
\label{Jor-admiss}
Let $(\mathcal{A},\cdot)$  be a shift associative algebra, i.e. an algebra
satisfying the identity  $(xy)z=y(zx)$. Then $(\mathcal{A}, \circ, [\cdot,\cdot])$ is an
anti-Poisson-Jordan algebra$;$
$(\mathcal{A}, \circ, [\cdot,\cdot])$ is an
anti-Poisson algebra if and only if $(\mathcal{A},\cdot)$  satisfies the identity $[[x,y],z]=0.$
\end{theorem}

\begin{definition} 
Let $({\mathcal A}, \cdot)$ be an algebra and $\delta$ be a fixed complex number.  
A linear map $\varphi$ is called a $\delta$-derivation if it satisfies
\begin{center}
$\varphi(x y) \ = \ \delta  \big(\varphi(x)y+ x \varphi(y)\big).$
\end{center}
\end{definition}

If $\delta=1$ (resp., $\delta=-1$), we have a derivation (resp., an antiderivation\footnote{Recall that the antiderivations play an important role in the definition of mock-Lie algebras \cite{Z17}.}).
The notion of $\delta$-derivations was introduced in a paper by Filippov \cite{fil1} $\big($see references in \cite{k23,DET,zz}$\big).$ 
Based on the principal relation between 
Poisson   (resp., anti-Poisson) algebras and derivations (resp., antiderivations), we introduce $\delta$-Poisson algebras, which are related in a similar way to $\delta$-derivations.

\begin{definition}
Let $\delta$ be a fixed complex number.
An algebra $({\rm P}, \cdot, \{\cdot,\cdot\})$ is called a {$\delta$-Poisson algebra}\footnote{For $\delta=1$, we have the definition of Poisson algebras.} 
if $({\rm P}, \cdot)$ is a  commutative associative   algebra,
$({\rm P},  \{\cdot,\cdot\})$ is a Lie algebra and the following
identity holds$:$
\begin{equation}\label{deltpois}
\{x  y, z \} \ =\  \delta \big(x  \{ y,z \} +\{x,z\}  y\big).
\end{equation}
 
\end{definition}

\begin{definition}
An algebra $({\rm P}, \cdot ,\{\cdot,\cdot\})$ is called a scalar-Poisson algebra if ${\rm P}$ is a $\delta$-Poisson algebra for every $\delta$. 
Equivalently,    the variety of scalar-Poisson algebras  is the intersection of the varieties of $\delta _{1}$%
-Poisson algebras and $\delta _{2}$-Poisson algebras for any distinct $\delta_{1}$ and $\delta_{2}$.
 Equivalently, an algebra $(\mathrm{P},\cdot ,\{\cdot ,\cdot \})$ is called a {%
scalar-Poisson algebra}, if $(\mathrm{P},\cdot )$ is a commutative
associative algebra, $(\mathrm{P},\{\cdot ,\cdot \})$ is a Lie algebra and the following identities hold: 
\begin{center}
$\{xy,z\} \  = \ 0, \ x\{y,z\}+\{x,z\}y \ =\ 0.$
\end{center}
\end{definition}

\begin{proposition}\label{antixyzt}
Let $\delta\neq1$ and let $({\rm P}, \cdot, \{\cdot,\cdot\})$ be a $\delta$-Poisson algebra. Then 
\begin{eqnarray}\label{antiP}
     \{xy,zt\} \ = \ \{x,yz\}t \ = \ \{x,yzt\} \ = \  \delta  \{x,y\}zt\ = \ \{ xy, \{z,t\}\} \ = \ 0,\\
 \label{cycl}    \{xy,z\} + \{yz,x\}+\{zx,y\} \ = \ 0.
\end{eqnarray}
\end{proposition}

\begin{proof}
If $\delta=0,$ then our statement is obvious.
Let us consider the case $\delta \neq0,1.$
It is easy to see that
\begin{center}\adjustbox{max width=\linewidth}{\begin{tabular}{rclcr}
$\{x, yzt\}$&$=$&$\delta\big(\{x,yz\}t+\{x,t\}yz\big)$&$=$&$\delta^2\{x,y\}zt+\delta^2\{x,z\}yt+\delta\{x,t\}yz,$\\
$\{x, yzt\}$&$=$&$\delta\big(\{x,y\}zt+\{x,zt\}y\big)$&$=$&$\delta\{x,y\}zt+\delta^2\{x,z\}yt+\delta^2\{x,t\}yz.$\\
\end{tabular}}\end{center}
 Subtracting, we have $\{x,y\}zt=\{x,t\}yz,$ 
 which gives 
 \begin{center}\adjustbox{max width=\linewidth}{\begin{tabular}{rcccccccccccl}
 $\{x,y\}zt$&$=$&$\{x,z\}yt$&$=$&$-\{z,x\}yt$&$=$&$-\{z,y\}xt$&$=$&$\{y,z\}xt$&$=$&$\{y,x\}zt.$
\end{tabular}}\end{center}
Hence,   $\{x,y\}zt=0$ and $\{x,yzt\} =0.$
Now, it is easy to see that 
\begin{center}\adjustbox{max width=\linewidth}{\begin{tabular}{lcl}
 $ \{xy,zt\} $&$ =$&$ \ \delta^2  \big( \{x,z\} yt+ \{y,z\}xt+\{x,t\}yz+\{y,t\}xz  \big)\ =\ 0,$\\  
 $\{x,yz\}t $&$ = $&$ \delta\big(\{x,y \}zt + \{x,z\} yt\big) \ =\ 0.$
\end{tabular}}\end{center} For the last equality in (\ref{antiP}), we use the following computation.

\begin{center}\adjustbox{max width=\linewidth}{\begin{tabular}{lcl}
$\{xy, \{z,t\}\}$ & $=$ & $ 
\{\{xy,z\}, t\} + \{ z, \{xy,t\}\}$ \\ 

& $=$ & $\delta \big( \{\{x,z\}y+x\{y,z\},t\} + \{z, \{x,t\}y+x\{y,t\}\} \big)$\\
& $=$ & $\delta^2 \Big( \{x,z \} \{y,t\}+\{\{x,z\},t\}y+ \{x,t\} \{y,z\}+x\{\{y,z\},t\}+$\\
\multicolumn{3}{r}{$ \{z,\{x,t\}\}y+\{z,y\}\{x,t\}+\{z,x\}\{y,t\}+x\{z,\{y,t\}\} \Big)$}\\

& $=$ & $\delta^2 \big(  \{x,\{z,t\}\}y + x\{y,\{z,t\}\} \big) \ = \ \delta    \{xy,\{z,t\}\}.$\\
\end{tabular}}\end{center}

\noindent
Hence, we have the last identity in (\ref{antiP}).
The identity (\ref{cycl}) follows from a direct calculation:
\begin{center}
$ \{xy,z\} + \{yz,x\}+\{zx,y\} \ = \ \delta \big(\{x,z\}y+x\{y,z\}+ \{y,x\}z+y\{z,x\}+\{z,y\}x+z\{x,y\}  \big)\ = \ 0.$
\end{center}
\end{proof}

Relation \eqref{antiP} immediately gives the following corollary.

\begin{corollary}
If $\delta  \neq 1,$ then there are no nontrivial\footnote{As usual, by a nontrivial Poisson-type algebra we mean an algebra with both multiplications nonzero.}  $\delta$-Poisson algebras with a unital commutative associative part.
\end{corollary}

\begin{theorem}\label{dPsimple}
    If $\delta \neq 1,$ then there are no nontrivial  simple $\delta$-Poisson algebras.
\end{theorem}

\begin{proof}
Let $({\rm P}, \cdot, \{\cdot, \cdot \})$ be a simple $\delta$-Poisson algebra.
Then ${\rm P}\cdot{\rm P}$ is an ideal of $({\rm P}, \cdot, \{\cdot, \cdot \}).$
It follows that $({\rm P}, \cdot)$ is perfect, i.e. ${\rm P} \cdot{\rm P}={\rm P}$.
Thanks to (\ref{antiP}), we obtain $\{ {\rm P},{\rm P}\}=0$ and the statement is proved.
\end{proof}

\begin{definition}
Let $({\rm P}, \cdot)$ be a commutative associative algebra
and $({\rm P}, \{\cdot,\cdot\})$ be an anticommutative algebra, then 
\begin{enumerate}
    \item[{\rm 1.}]  $({\rm P}, \cdot, \{\cdot,\cdot\})$ is an $F$-manifold algebra if $({\rm P}, \{\cdot,\cdot\})$ is a Lie algebra and the following identity holds: 
\begin{center}\adjustbox{max width=\linewidth}{\begin{tabular}{rcl}
   $\{x  y, z  t \}$&$=$&$ \{x  y, z \}  t + 
\{x  y, t \}  z + x  \{y, z  t\} + y  \{x, z  t\}- \big(x  z  \{y, t \} + y  z  \{x, t\} + 
y  t  \{x, z\} + x  t  \{y, z\}\big).$
\end{tabular}}\end{center}
 \item[{\rm 2.}]  $({\rm P}, \cdot, \{\cdot,\cdot\})$ is an algebra of Jordan brackets if the following identities hold: 
 
\begin{center}\adjustbox{max width=\linewidth}{\begin{tabular}{rcl}
 $\{\{x , y\}  z , t\} + \{\{y, t\}   z , x\}+ \{\{t, x\}  z , y\} $&$=$&$ 
 \{x , y\}  \{z , t\}+ \{y, t\}  \{z , x\} + \{t , x\}  \{z , y\}, $\\
$ \{y  t , z\} x +\{x,z\}  y t $&$=$&$ \{t x , z\} y + \{y,z\} t x,$\\ 
 $\{t x , y z\}+\{t y , x z\}+\{x  y z , t\}$&$ =
 $&$\{ t y , z\} x   + \{t x , z\} y+ x y \{z , t\}.$
\end{tabular}}\end{center}

\end{enumerate}

\end{definition}

The varieties of $F$-manifold algebras and algebras of Jordan brackets 
are principal varieties that generalize the varieties of both Poisson and transposed Poisson algebras \cite{kms}.

\begin{theorem} 
 Let $\delta\notin \big\{0,1\big\}$ and let $({\rm P}, \cdot, \{\cdot,\cdot\})$ be a $\delta$-Poisson algebra. Then 

 \begin{enumerate}

          \item [{\rm 1.}]  $({\rm P}, \cdot, \{\cdot,\cdot\})$ is an $F$-manifold algebra.

          \item[{\rm 2.}]    $({\rm P}, \cdot, \{\cdot,\cdot\})$ is an algebra of Jordan brackets if and only if it satisfies 
      \begin{center}
             $\{x,y \}\{z,t \} +\{y,t \}\{z,x \}+\{t,x \}\{z,y \} \ =\ 0$\footnote{This identity plays an important role in the theory of Poisson algebras; it is called the strong identity \cite{TP1}. Note that each transposed $\delta$-Poisson algebra also satisfies this identity (see Proposition \ref{idtdpa}).}.
      \end{center}

 \end{enumerate}  
\end{theorem}

\begin{proof}
First, due to Proposition \ref{antixyzt}, all terms of the identity of $F$-manifold algebras vanish, so $({\rm P}, \cdot, \{\cdot,\cdot\})$ is an $F$-manifold algebra.
Second, due to Proposition \ref{antixyzt}, it satisfies the second and the third identities of Jordan brackets, and, by \eqref{deltpois} and the Jacobi identity, the difference of the two sides of the first identity of Jordan brackets equals $\delta-1$ times the left-hand side of the strong identity. This gives our statement.
\end{proof}

The next statement can be obtained by straightforward calculations.

 \begin{proposition}\label{deltaxyzt}
    
 Let  $({\rm P}, \cdot, \{\cdot,\cdot\})$ be a $0$-Poisson algebra. Then 

 \begin{enumerate}

          \item[{\rm 1.}]    $({\rm P}, \cdot, \{\cdot,\cdot\})$ is an $F$-manifold algebra if and only if
        it satisfies 
        \begin{center}
        $x  z  \{y, t \} + y  z  \{x, t\} + 
y  t  \{x, z\} + x  t  \{y, z\}=0.$\end{center}

          \item[{\rm 2.}]     $({\rm P}, \cdot, \{\cdot,\cdot\})$ is an algebra of Jordan brackets if and only if it satisfies 
      \begin{center}
            $\{x,y \}\{z,t \} +\{y,t \}\{z,x \}+\{t,x \}\{z,y \} \ =\ 0$ \ and \ $xy\{z,t\}\ =\ 0$.
      \end{center}

 \end{enumerate}  
\end{proposition}

\noindent The second identity in part 2 cannot be omitted: the algebra with the multiplication tables $u\cdot u=u$, $u\cdot c=c$ and $\{a,b\}=c$ is a $0$-Poisson algebra satisfying the strong identity, but $uu\{a,b\}=c\neq0$, so it is not an algebra of Jordan brackets.

\begin{corollary}
    Let $\mathcal A$ be a cyclic associative algebra, 
    i.e., it satisfies the identities 
\begin{center}
        $(xy)z=(yz)x=x(yz).$
\end{center}   Then $(\mathcal A, \circ, [\cdot, \cdot])$ is 
     a scalar-Poisson algebra, an $F$-manifold algebra, and an algebra of Jordan brackets.
\end{corollary}
\begin{proof}
Thanks to Theorem \ref{Jor-admiss},
$(\mathcal A, \circ, [\cdot, \cdot])$ is anti-Poisson-Jordan.
It is easy to see that 
\begin{center}\adjustbox{max width=\linewidth}{\begin{tabular}{rclcl}
$[[x,y],z] $&$=$&$ xyz-yxz-zxy+zyx$&$=$&$0,$ \\
$x\circ[y,z]+[x,z]\circ y $&$=$&$ xyz-xzy+xzy-zxy$&$=$&$0.$
\end{tabular}}\end{center}
\noindent Then, due to Theorem \ref{Jor-admiss}, it is anti-Poisson
and $[x \circ y,z]=0$. 
Hence, $(\mathcal A, \circ, [\cdot, \cdot])$ is a scalar-Poisson algebra, an
$F$-manifold algebra, and an algebra of Jordan brackets.\end{proof}

Let us recall a classical result on Poisson algebras.
\begin{proposition}\label{tensor} Let $({\rm P}_1, \cdot_1, \{\cdot ,\cdot \}_1)$ and $({\rm P}_2, \cdot_2, \{\cdot ,\cdot\}_2)$ be two Poisson algebras. Define two
operations $\cdot$ and $\{\cdot ,\cdot \}$ on ${\rm P}_1 \otimes {\rm P}_2$ by
\begin{center}\adjustbox{max width=\linewidth}{\begin{tabular}{rcl}
$(x_1 \otimes y_1) \cdot (x_2 \otimes y_2)$ &$ =$&$ x_1 \cdot_1 x_2\otimes y_1 \cdot_2 y_2,$\\
$\{x_1 \otimes y_1, x_2 \otimes  y_2\}$&$= $&$
\{x_1, x_2\}_1 \otimes y_1 \cdot_2 y_2 + x_1 \cdot_1 x_2 \otimes \{y_1, y_2\}_2.$
\end{tabular}}\end{center}\noindent 
Then $({\rm P}_1 \otimes {\rm P}_2, \cdot, \{\cdot,\cdot\})$  is a Poisson algebra.
\end{proposition}

This property is also shared by 
transposed Poisson  \cite{TP1} and 
Novikov-Poisson \cite{xu} algebras.
On the other hand, 
$F$-manifold algebras and contact bracket algebras, in general,  do not have this property \cite{LSB, zus}.
We show that this property is also shared by $\delta$-Poisson algebras.

\begin{theorem}\label{tensdp}
    Let $({\rm P}_1, \cdot_1, \{\cdot ,\cdot \}_1)$ and $({\rm P}_2, \cdot_2, \{\cdot ,\cdot\}_2)$ be two $\delta$-Poisson algebras. Define two
operations as in Proposition \ref{tensor}.
Then $({\rm P}_1 \otimes {\rm P}_2, \cdot, \{\cdot,\cdot\})$  is a $\delta$-Poisson algebra.
\end{theorem}

\begin{proof}
    Let us consider $x_1, x_2,x_3 \in {\rm P}_1$ and $y_1,y_2,y_3 \in {\rm P}_2,$ then 

\begin{center}\adjustbox{max width=\linewidth}{\begin{tabular}{lcccl}
$\{(x_1 \otimes y_1) \cdot (x_2 \otimes y_2), x_3 \otimes y_3 \}$&$
= $&$\{x_1 \cdot_1 x_2 \otimes y_1 \cdot_2 y_2  , x_3 \otimes y_3 \} $&$=$ \\ 
\multicolumn{3}{c}{
$\{x_1 \cdot_1 x_2, x_3\}_1 \otimes  (y_1\cdot_2 y_2 \cdot_2 y_3) + (x_1\cdot_1 x_2\cdot_1 x_3)  \otimes \{y_1\cdot_2 y_2, y_3 \}_2$}&$ =$ \\
\multicolumn{3}{l}{$\delta\Big( \big(\{x_1 , x_3\}_1 \cdot_1 x_2 + x_1 \cdot_1\{ x_2, x_3\}_1\big) \otimes  (y_1\cdot_2 y_2 \cdot_2 y_3)+$}\\
\multicolumn{3}{r}{$(x_1\cdot_1 x_2\cdot_1 x_3)  \otimes \big(y_1\cdot_2\{ y_2, y_3 \}_2+\{y_1, y_3 \}_2\cdot_2 y_2\big) \Big)$} & $=$\\

\multicolumn{3}{l}{$\delta\Big( 
\big( \{x_1 , x_3\}_1 \otimes (y_1 \cdot_2 y_3) +(x_1 \cdot_1 x_3 ) \otimes \{y_1, y_3\}_2 \big) \cdot(x_2\otimes y_2)+$}\\
\multicolumn{3}{r}{$(x_1 \otimes y_1) \cdot \big(\{x_2, x_3\}_1 \otimes (y_2 \cdot_2  y_3)+
(x_2 \cdot_1 x_3) \otimes \{ y_2, y_3\}_2  \big)   \Big)$} & $=$\\

\multicolumn{3}{c}{$\delta\big(  \{x_1 \otimes  y_1, x_3 \otimes y_3\} \cdot (x_2 \otimes  y_2) +  (x_1 \otimes y_1) \cdot  \{x_2 \otimes y_2, x_3 \otimes  y_3\} \big).$}  \\
\end{tabular}}\end{center}
Hence, (\ref{deltpois}) holds. For the Jacobi identity, we have

{\begin{longtable}{lccccccc}
$\{\{x_1 \otimes y_1, x_2 \otimes y_2\}, x_3 \otimes y_3\}$ &$=$&
$\{\{x_1,x_2\}_1 \otimes (y_1 \cdot_2 y_2) + (x_1 \cdot_1 x_2) \otimes \{y_1, y_2\}_2, x_3 \otimes y_3\}$ &$=$\\
\multicolumn{3}{l}{$\{\{x_1,x_2\}_1,  x_3\}_1 \otimes (y_1 \cdot_2 y_2\cdot_2 y_3) +
 (\{x_1,x_2\}_1 \cdot_1 x_3) \otimes \{ y_1 \cdot_2 y_2, y_3\}_2 +$}\\ 
\multicolumn{3}{r}{$ \{x_1 \cdot_1 x_2, x_3\}_1 \otimes (\{y_1, y_2\}_2 \cdot_2 y_3)+
 (x_1 \cdot_1 x_2\cdot_1 x_3) \otimes \{\{y_1, y_2\}_2, y_3\}_2
 $} &$=$\\
\multicolumn{3}{l}{$\{\{x_1,x_2\}_1,  x_3\}_1 \otimes (y_1 \cdot_2 y_2\cdot_2 y_3) +$}\\
\multicolumn{3}{c}{
$\delta\big((\{ x_1,x_2\}_1 \cdot_1 x_3) \otimes (\{ y_1 , y_3\}_2\cdot_2 y_2) +
 (\{x_1,x_2\}_1 \cdot_1 x_3) \otimes  (y_1\cdot_2 \{ y_2 , y_3\}_2)+  $}\\ 
\multicolumn{3}{c}{
$ (\{x_1 , x_3\}_1  \cdot_1 x_2)\otimes (\{y_1, y_2\}_2 \cdot_2 y_3)+
 (x_1 \cdot_1 \{  x_2, x_3\}_1) \otimes (\{y_1, y_2\}_2 \cdot_2 y_3) \big)+$}\\
\multicolumn{5}{r}{$(x_1 \cdot_1 x_2\cdot_1 x_3) \otimes \{\{y_1, y_2\}_2, y_3\}_2
 .$} \\
\end{longtable}} 

Hence, 
\begin{center}
    $\{\{x_1 \otimes y_1, x_2 \otimes y_2\}, x_3 \otimes y_3\}+ 
    \{\{x_2 \otimes y_2, x_3 \otimes y_3\}, x_1 \otimes y_1\}+
    \{\{x_3 \otimes y_3, x_1 \otimes y_1\}, x_2 \otimes y_2\}\ = \ 0.$
\end{center}
\end{proof}

\begin{definition}
Let $\delta$ be a fixed element from the basic field.
An algebra $({\rm P},  \{\cdot,\cdot\})$ is a  $\delta$-Poisson structure\footnote{$\delta$-Poisson structures are related to 
post-Lie algebras \cite{bu} and transposed Poisson algebras \cite{TP1}.} on a  (not necessarily commutative or associative) algebra
$({\rm P}, \cdot),$ if the two multiplications $\cdot$ and $\{\cdot,\cdot\}$ satisfy \eqref{deltpois}\footnote{In the noncommutative case, one has to add the identity 
 $\{x , y z \} = \delta\big(y  \{ x,z \} +\{x,y\}  z\big).$}.
\end{definition}

\begin{definition}
Let $\delta\notin \big\{ 0,1\big\}$ be a fixed element from the basic field and  $({\rm P},  \cdot)$ an algebra.
We say that $({\rm P}, \cdot)$ has only trivial $\delta$-derivations if
 there is a basis $\{ e_i \}_{i\in I}$ such that for each $\delta$-derivation $\varphi$ 
 there are elements $\{\alpha_i \} _{i\in I}$ from the basic field, 
 such that $\varphi(e_i)= \alpha_i e_i,$ for all $i \in I.$  
\end{definition}

\begin{proposition}\label{propdelta}
Let  $\delta\neq 0,1$ be a fixed element from the basic field and $({\rm P},  \cdot)$ an algebra of dimension $\geq 2.$
If each  
$\delta$-derivation of the algebra  $({\rm P},  \cdot)$   is   trivial, then $({\rm P},  \cdot)$ does not admit nontrivial $\delta$-Poisson structures.
\end{proposition}

 \begin{proof}
 Let us assume that 
$({\rm P},  \cdot)$  admits a nontrivial $\delta$-Poisson structure $({\rm P}, \{\cdot,\cdot\})$.
Then each left multiplication $\mathfrak L_x$ (resp., right multiplication $\mathfrak R_x$) by an element $x$ in the algebra $({\rm P}, \{\cdot,\cdot\})$ is a $\delta$-derivation $\varphi_x$ (resp., $\delta$-derivation $\psi_x$) of the algebra $({\rm P},  \cdot).$ 
If $({\rm P},  \cdot)$ does not have nontrivial $\delta$-derivations, then there is  a basis $\{ e_i \}_{i\in I}$ such that for each $\delta$-derivation $\phi$ 
 there are elements $\{\alpha_i \} _{i\in I}$ from the basic field, 
 such that $\phi(e_i)= \alpha_i e_i,$ for all $i \in I.$  
Let us consider two different elements $e_i$ and $e_j$ from the basis. Hence, 
\begin{center}
$\langle e_j \rangle \ni \varphi_{e_i}(e_j) = \{ e_i,e_j\} = \psi_{e_j}(e_i) \in \langle e_i \rangle,$
\end{center}
 i.e., $\{ e_i,e_j\}=0$ and $({\rm P}, \{\cdot,\cdot\})$ is trivial.
\end{proof}

\begin{remark}
The notions of $\delta$-Poisson algebras, $\delta$-Poisson structures, 
 algebras with trivial $\delta$-derivations, and the statement of Proposition \ref{propdelta}  
can be extended to the super and $n$-ary cases in the usual way.
\end{remark}

\begin{proposition}\label{tpstr}
    Let $\delta \notin \big\{ 0,1\big\}$ and $({\rm P}, \cdot)$ be one of the following algebras
    \begin{enumerate}
        \item[{\rm 1.}] 
    a unital algebra over a field of characteristic $\neq 2$, except for the case $\delta=\frac{1}{2};$
        
    \item[{\rm 2.}] complex semisimple finite-dimensional associative algebra$;$
    
    \item[{\rm 3.}] complex semisimple finite-dimensional structurable algebra$;$
    
    \item[{\rm 4.}] complex semisimple finite-dimensional Jordan algebra or  superalgebra$;$
    
    \item[{\rm 5.}] complex simple finite-dimensional Lie algebra or    superalgebra, 
    except for the case of $\mathfrak{sl}_2$ and $\delta=-1;$
   
    \item[{\rm 6.}] perfect commutative associative algebra.

    \end{enumerate}
    Then $({\rm P}, \cdot)$ does not admit nontrivial $\delta$-Poisson structures.
\end{proposition}

 \begin{proof}
 The statement follows from Proposition \ref{propdelta}.
 Namely,\begin{enumerate}
     \item 
  thanks to \cite{kay07}, each unital algebra over a field of characteristic $\neq 2$ has only trivial $\delta$-derivations if $\delta \neq \frac{1}{2}$;
 \item  thanks to \cite{sh12}, each complex semisimple finite-dimensional associative algebra has only trivial $\delta$-derivations;
 \item  thanks to \cite{K014}, each complex semisimple finite-dimensional structurable algebra has only trivial $\delta$-derivations;
\item  thanks to \cite{kay10,kay12mz}, each complex semisimple finite-dimensional Jordan algebra or  superalgebra has only trivial $\delta$-derivations;
 \item thanks to \cite{kay09,kay10}, each complex simple finite-dimensional Lie   superalgebra has only trivial $\delta$-derivations;

 \item thanks to \cite{fil1}, each complex simple finite-dimensional Lie algebra, except for the case $\mathfrak{sl}_2$ and $\delta=-1,$ has only trivial $\delta$-derivations.

 \end{enumerate}
\noindent 
     For perfect commutative associative algebras, we use relation (\ref{antiP}) and obtain that  $({\rm P}, \{\cdot,\cdot\})$ is trivial.
         \end{proof}

\begin{proposition}
    Let  $({\rm P}, \cdot)$ be a perfect algebra (superalgebra, $n$-ary algebra), 
    i.e. ${\rm P} \cdot{\rm P}= {\rm P}$. 
    Then $({\rm P}, \cdot)$ does not admit nontrivial $0$-Poisson structures.
\end{proposition}

\subsection{Transposed $\delta$-Poisson algebras} 

\begin{definition}
Let $\delta$ be a fixed complex number.
An algebra $({\rm P}, \cdot, \{\cdot,\cdot\})$ is called a {transposed $\delta$-Poisson algebra}\footnote{For $\delta=\frac{1}{2}$, we have the standard notion of transposed Poisson algebras, introduced in \cite{TP1}.}\footnote{
A similar notion  can be found in \cite{R22},
where an algebra $({\rm P}, \cdot, \{\cdot,\cdot\})$
is called an anti-Poisson algebra, if 
$({\rm P}, \cdot)$ is mock-Lie, 
$({\rm P}, \{\cdot,\cdot\})$ is anticommutative and it satisfies \eqref{deltranspois} for $\delta=-1$.
}, 
if $({\rm P}, \cdot)$ is a  commutative associative    algebra,
$({\rm P},  \{\cdot,\cdot\})$ is a Lie algebra and the following
identity holds:
\begin{equation}\label{deltranspois}
x\{  y, z \} = \delta \big(  \{x y,z \} +\{y,xz\}  \big).
\end{equation}
 
\end{definition}

 \begin{proposition}[see, \cite{R22}]
Let $({\rm A}, \cdot)$ be an antiassociative algebra 
$($i.e., $({\rm A}, \cdot)$ satisfies the identity $(xy)z+x(yz)=0),$ then 
\begin{center}\adjustbox{max width=\linewidth}{\begin{tabular}{lcl}
$x \circ [  y, z ] $&$=$&$ - \big(  [x \circ y,z ] +[y,x \circ z]  \big).$
\end{tabular}}\end{center}\noindent
In particular, $({\rm A}, \circ, [\cdot,\cdot])$ is a generic transposed anti-Poisson-Jacobi-Jordan algebra.
 \end{proposition}

\begin{proposition}
If $({\rm P}, \cdot, \{\cdot,\cdot\})$ is a nontrivial transposed $\delta$-Poisson algebra
with a unital algebra $({\rm P}, \cdot),$
then $\delta=\frac{1}{2}.$
\end{proposition}
\begin{proof}
    Taking $x=1$ in \eqref{deltranspois}, we obtain 
    $\{  y, z \} = 2\delta \{y,z \}$ and $\delta=\frac{1}{2}.$
\end{proof}

The next statement is a generalization of \cite[Theorem 2.7]{TP1} from the case $\delta=\frac{1}{2}$ to an arbitrary $\delta$.
The proof is a step-by-step repetition of the proof of \cite[Theorem 2.7]{TP1} with $\frac{1}{2}$ replaced by $\delta$.

\begin{proposition}\label{idtdpa}
    Let $({\rm P}, \cdot, \{\cdot, \cdot\})$ be a transposed $\delta$-Poisson algebra. Then it satisfies the identities \eqref{idtp1} and \eqref{idtp3}--\eqref{idtp6}$;$ if $\delta\neq0$, then it also satisfies \eqref{idtp2}$:$
    \begin{eqnarray} \label{idtp1}
  \{x,y\}z+\{y,z\}x+\{z,x\}y \ =\ 0,\\
  \label{idtp2}
  \{hx, \{y,z\}\}  + \{hy, \{z,x\}\}  + \{hz, \{x, y\}\} = 0, \\  
    \label{idtp3}
 \{ h \{x, y\}, z\}+ \{h \{y,z\}, x\}+\{h\{z,x\}, y\}=0,\\  \label{idtp4}
 \{   h, x\}\{y,z\}+\{h, y\}\{z,x\}+\{h, z\}\{x, y\}=0, \\  \label{idtp5}
\delta^{2}(\{xu, yv\}+\{xv, yu\})=(1-\delta)uv \{x, y\},\\ \label{idtp6}
\delta(x\{u, yv\} + v\{xy, u\}) + (1-\delta)yu\{v,x\}=0.
    \end{eqnarray}
\end{proposition}

\noindent For $\delta=0$, the identity \eqref{idtp2} does not hold in general: let ${\rm P}=\mathfrak{sl}_2\oplus\mathbb{C}e$, where $\mathfrak{sl}_2$ is the Lie part, $e$ is central with respect to $\{\cdot,\cdot\}$, and the only nonzero product is $e\cdot e=f\neq0$, $f\in\mathfrak{sl}_2$; then ${\rm P}$ is a transposed $0$-Poisson algebra and, for $h=x=e$, the left-hand side of \eqref{idtp2} equals $\{f,\{y,z\}\}$, which is nonzero for suitable $y,z$.

\begin{proposition}\label{dptm}
    Let $\delta_1\delta_2\neq \frac{1}{3}$ and let $({\rm P}, \cdot, \{\cdot, \cdot\})$ be a $\delta_1$-Poisson and transposed $\delta_2$-Poisson algebra. Then it is a mixed-Poisson algebra, i.e. it satisfies the following identities
    \begin{center}
         $\{x,y\}z \ = \ \{x  y, z\} \ = \ 0.$
    \end{center}
\end{proposition}

\begin{proof}
    It is easy to see that

    \begin{center}\adjustbox{max width=\linewidth}{\begin{tabular}{lclcccc}
    $0$ &$=$ &
    $\delta_2 \big( \{ zx,y\}+ \{x,zy\} \big) - z\{x,y\} \ = $\\
    && $\delta_1\delta_2 \big( \{z,y\}x+ \{x,y\}z+\{x,z\}y+\{x,y\}z \big) - z\{x,y\} $ &$=$\\
    \multicolumn{3}{r}{$\delta_{1}\delta_2 \big(\{z,y\}x+ \{x,z\}y \big)+ (2\delta_1\delta_2-1) \{x,y\}z$} & $ \overset{(\ref{idtp1})}{ =}$ &$ (3\delta_1\delta_2-1) \{x,y\}z.$
\end{tabular}}\end{center}
    Hence, $\{x,y\}z \ =\ 0$ and $\{xy,z\} \ \overset{(\ref{deltpois})}{ =} \ 0.$ 
\end{proof}

\begin{example}\label{delta-mix-ex}
Consider the following $5$-dimensional Poisson-type algebras$:$
\[\adjustbox{max width=\linewidth}{$\displaystyle {\rm P}^{\beta \neq 0}:=\left\{ 
\begin{tabular}{lll}
$e_{1}\cdot e_{1}=e_{3},$ & $e_{1}\cdot e_{2}=e_{5},$ & $e_{1}\cdot
e_{3}=e_{4},$ \\ 
$\{e_{1},e_{2}\}=3\beta e_{3},$ & $\{e_{1},e_{5}\}=e_{4},$ & $%
\{e_{2},e_{3}\}=-2e_{4}.$%
\end{tabular}%
\right.$}\]%
Set $\delta _{1}=\frac{1}{3\beta }$ and $\delta _{2}=\beta $. Then, $%
\rm{P}^{\beta }$ is a $\delta _{1}$-Poisson algebra and transposed $\delta _{2}$-Poisson algebra. Moreover, $\delta_1\delta_2= \frac{1}{3}$ and $%
\rm{P}^{\beta }$ is not a mixed-Poisson algebra.
\end{example}
\begin{theorem}\label{tenstdp}
    Let $({\rm P}_1, \cdot_1, \{\cdot ,\cdot \}_1)$ and $({\rm P}_2, \cdot_2, \{\cdot ,\cdot\}_2)$ be  transposed $\delta$-Poisson algebras. Define two
operations as in Proposition \ref{tensor}.
Then $({\rm P}_1 \otimes {\rm P}_2, \cdot, \{\cdot,\cdot\})$  is a transposed  $\delta$-Poisson algebra.
\end{theorem}

\begin{proof}
    Let us consider $x_1, x_2,x_3 \in {\rm P}_1$ and $y_1,y_2,y_3 \in {\rm P}_2,$ then 

\begin{center}\adjustbox{max width=\linewidth}{\begin{tabular}{lcccl}
$\{x_1 \otimes y_1 , x_2 \otimes y_2\} \cdot (x_3 \otimes y_3) \
= $ \\
$\{x_1 , x_2 \}_1 \cdot_1 x_3 \otimes (y_1 \cdot_2 y_2 \cdot_2 y_3)+
(x_1 \cdot_1 x_2 \cdot_1 x_3) \otimes \{y_1 ,  y_2 \}_2 \cdot_2 y_3  $&$=$ \\ 
\multicolumn{1}{l}{
$\delta\big(
\{x_1 \cdot_1 x_3, x_2 \}_1  \otimes (y_1 \cdot_2 y_2 \cdot_2 y_3)+
\{x_1 , x_2 \cdot_1 x_3\}_1  \otimes (y_1 \cdot_2 y_2 \cdot_2 y_3)+$}\\
\multicolumn{1}{r}{
$
(x_1 \cdot_1 x_2 \cdot_1 x_3) \otimes \{y_1\cdot_2 y_3 ,  y_2 \}_2   +
(x_1 \cdot_1 x_2 \cdot_1 x_3) \otimes \{y_1 ,  y_2\cdot_2 y_3 \}_2 \big)$}&$=$ \\ 
\multicolumn{3}{r}{$\delta\big(
\{(x_1 \otimes y_1)\cdot (x_3\otimes y_3),  x_2 \otimes y_2\}+  
\{x_1 \otimes y_1,  (x_2 \otimes y_2)\cdot (x_3\otimes y_3)\} \big).$ }
\end{tabular}}\end{center}
Hence, (\ref{deltranspois}) holds. For the Jacobi identity, we have

\begin{longtable}{lccccccc}
$\{\{x_1 \otimes y_1, x_2 \otimes y_2\}, x_3 \otimes y_3\}$ &$=$&
$\{\{x_1,x_2\}_1 \otimes (y_1 \cdot_2 y_2) + (x_1 \cdot_1 x_2) \otimes \{y_1, y_2\}_2, x_3 \otimes y_3\}$ &$=$\\
\multicolumn{3}{l}{$\{\{x_1,x_2\}_1,  x_3\}_1 \otimes (y_1 \cdot_2 y_2\cdot_2 y_3) +
( \{x_1,x_2\}_1 \cdot_1 x_3) \otimes \{ y_1 \cdot_2 y_2, y_3\}_2 +$}\\ 
\multicolumn{3}{r}{$ \{x_1 \cdot_1 x_2, x_3\}_1 \otimes (\{y_1, y_2\}_2 \cdot_2 y_3)+
 (x_1 \cdot_1 x_2\cdot_1 x_3) \otimes \{\{y_1, y_2\}_2, y_3\}_2
 $} &$=$\\

\multicolumn{3}{l}{$\{\{x_1,x_2\}_1,  x_3\}_1 \otimes (y_1 \cdot_2 y_2\cdot_2 y_3) +$}\\
\multicolumn{3}{c}{
$\delta\big(
 \{x_1\cdot_1 x_3,x_2\}_1  \otimes \{ y_1 \cdot_2 y_2, y_3\}_2 +
 \{x_1,x_2\cdot_1 x_3\}_1  \otimes \{ y_1 \cdot_2 y_2, y_3\}_2 +
  $}\\ 
\multicolumn{3}{c}{
$ \{x_1  \cdot_1 x_2, x_3\}_1 \otimes \{y_1\cdot_2 y_3, y_2\}_2  + 
 \{x_1   \cdot_1 x_2, x_3\}_1 \otimes \{y_1,  y_2\cdot_2 y_3\}_2 \big) + $}\\
\multicolumn{5}{r}{$  
 (x_1 \cdot_1 x_2\cdot_1 x_3) \otimes \{\{y_1, y_2\}_2, y_3\}_2
 .$}
\end{longtable} 
Hence, 
\begin{center}
    $\{\{x_1 \otimes y_1, x_2 \otimes y_2\}, x_3 \otimes y_3\}+ 
     \{\{x_2 \otimes y_2, x_3 \otimes y_3\}, x_1 \otimes y_1\}+
     \{\{x_3 \otimes y_3, x_1 \otimes y_1\}, x_2 \otimes y_2\}\ = \ 0.$
\end{center}
\end{proof}

\begin{proposition} 
 Let $\delta\notin \big\{ \frac{1}{2},\frac{1}{3},1 \big\}$ and let $({\rm P}, \cdot, \{\cdot,\cdot\})$ be a transposed $\delta$-Poisson algebra. Then 
    $({\rm P}, \cdot, \{\cdot,\cdot\})$ is an $F$-manifold algebra if and only if it satisfies  $\{xy,zt \} \  =\ 0$.
  \end{proposition} 
\begin{proof}
If  $({\rm P}, \cdot, \{\cdot,\cdot\})$ is an $F$-manifold algebra, then we have     
\begin{center}\adjustbox{max width=\linewidth}{\begin{tabular}{lcl}
$\{xy,zt\}$
&$=$&$\{xy,z\}t+\{xy,t\}z+x\{y,zt\}+y\{x,zt\}-xz\{y,t\}-yz\{x,t\}-yt\{x,z\}-xt\{y,z\}$ \\
&$ \overset{(\ref{deltranspois},\ref{idtp5})}{ =}$&$\delta \Bigl(
\{txy,z\}+\{xy,tz\}+\{zxy,t\}+\{xy,zt\}+\{xy,zt\}+\{y,xzt\}+\{yx,zt\}+%
\{x,yzt\}\Bigr)$  \\
&&$+\frac{\delta ^{2}}{\delta -1}\Bigl(
\{xy,zt\}+\{zy,xt\}+\{yx,zt\}+\{zx,yt\}+$\\
 \multicolumn{3}{r}{$\{yx,tz\}+\{tx,yz\}+\{xy,tz\}+%
\{ty,xz\}\Bigr)$}  \\
&$=$&$\delta \Bigl( \{txy,z\}+\{zxy,t\}+\{y,xzt\}+\{x,yzt\}\Bigr) + 4\delta
  \{xy,zt\} +\frac{4\delta ^{2}}{
\delta -1}\{xy,zt\}$  \\
&$ \overset{(\ref{deltranspois} )}{ =}$&$tx\{y,z\}+zy\{x,t\}+4\delta \{xy,zt\}+\frac{4\delta ^{2}}{\delta -1}
\{xy,zt\}$ \\
&$ \overset{( \ref{idtp5})}{ =}$&$\frac{\delta ^{2}}{1-\delta }\{ty,xz\}+\frac{\delta ^{2}}{1-\delta }
\{xy,tz\}+\frac{\delta ^{2}}{1-\delta }\{zx,yt\}+\frac{\delta ^{2}}{1-\delta 
}\{yx,zt\}+4\delta \{xy,zt\}+\frac{4\delta ^{2}}{\delta -1}\{xy,zt\}$ \\
&$=$&$\frac{2\delta \left( 3\delta -2\right) }{\delta -1}\{xy,zt\}.$
\end{tabular}}\end{center}
Thus 
$\left( 6\delta ^{2}-5\delta +1\right) \{xy,zt\}=0$. Since  $\delta\notin \big\{ \frac{1}{2},\frac{1}{3} \big\}$, we have $\{xy,zt\}=0$.
On the other hand, if $({\rm P}, \cdot, \{\cdot,\cdot\})$ satisfies the identity $\{xy,zt\}=0$, then, from \eqref{idtp5}, it satisfies the identity $uv\{x,y\}=0$. Whence
\begin{center}\adjustbox{max width=\linewidth}{\begin{tabular}{lcl}
$\{xy,zt\}$ &$=$&$0$ \\
&$=$&$tx\{y,z\}+zy\{x,t\}+4\delta \{xy,zt\}$ \\
&$\overset{(\ref{deltranspois} )}{ =}$&$\delta \Bigl( \{txy,z\}+\{zxy,t\}+\{y,xzt\}+\{x,yzt\}  +  \{xy,zt\}+\{xy,zt\}+\{xy,tz\}+\{yx,zt\}\Bigr)$  \\
&$=$&$\delta \Bigl(
\{txy,z\}+\{xy,tz\}+\{zxy,t\}+\{xy,zt\}+\{xy,zt\}+\{y,xzt\}+\{yx,zt\}+%
\{x,yzt\}\Bigr)$  \\
&$\overset{\eqref{deltranspois}}{ =}$&$t\{xy,z\}+z\{xy,t\}+x\{y,zt\}+y\{x,zt\}$ \\
&$=$&$t\{xy,z\}+z\{xy,t\}+x\{y,zt\}+y\{x,zt\}-xz\{y,t\}-yz\{x,t\}-yt\{x,z\}-xt%
\{y,z\}.$
\end{tabular}}\end{center}
Therefore, $({\rm P}, \cdot, \{\cdot,\cdot\})$ is an $F$-manifold algebra.
\end{proof}

\begin{example}
Consider the following $5$-dimensional Poisson-type algebra$:$%
\[\adjustbox{max width=\linewidth}{$\displaystyle \rm{P}:=\left\{ 
\begin{tabular}{llllll}
$e_{1}\cdot e_{1}=e_{2},$ &  $e_{1}\cdot e_{3}=e_{4},$ & $e_{1}\cdot e_{2}=e_{3},$ \\ 
$e_{1}\cdot e_{4}=e_{5},$ & $e_{2}\cdot e_{2}=e_{4},$ & $e_{2}\cdot e_{3}=e_{5},$ \\ 
$\{e_{1},e_{2}\}=\frac{1}{3}e_{3},$ & $\{e_{1},e_{3}\}=e_{4},$ & $%
\{e_{1},e_{4}\}=2e_{5},$ & $\{e_{2},e_{3}\}=e_{5}.$ &  & 
\end{tabular}%
\right.$}\]%
 $\rm{P}$ is a transposed $\frac{1}{3}$-Poisson algebra and an $F$-manifold algebra, but it does not satisfy $\{xy,zt\}= 0$.
\end{example}

\begin{example}
\bigskip Consider the following $3$-dimensional Poisson-type algebra$:$%
\[\adjustbox{max width=\linewidth}{$\displaystyle \rm{P}:=\left\{ 
\begin{tabular}{lll}
$e_{1}\cdot e_{1}=e_{1},$ & $e_{1}\cdot e_{2}=e_{2},$ & $e_{1}\cdot
e_{3}=e_{3},$ \\ 
$\{e_{1},e_{3}\}=e_{2}.$ &  & 
\end{tabular}%
\right.$}\]%
$\rm{P}$ is a transposed $\frac{1}{2}$-Poisson algebra and an $F$-manifold algebra, but it does not satisfy $\{xy, zt \}=  0$.
\end{example}

 \begin{proposition} 
 Let $({\rm P}, \cdot, \{\cdot,\cdot\})$ be a transposed $1$-Poisson algebra. If 
    $({\rm P}, \cdot, \{\cdot,\cdot\})$ is an $F$-manifold algebra, then it satisfies  $\{xy,zt \} \  =\ 0$.
  \end{proposition} 

  \begin{proof}
  Let $({\rm P}, \cdot, \{\cdot,\cdot\})$ be a transposed $1$-Poisson algebra and an $F$-manifold algebra. Then we have%
\[\adjustbox{max width=\linewidth}{$\begin{array}{rcl}\{xy,zt\}
&=&\{xy,z\}t+\{xy,t\}z+x\{y,zt\}+y\{x,zt\}-xz\{y,t\}-yz\{x,t\}-yt\{x,z\}-xt%
\{y,z\} \\
&=&\{xy,z\}t+\{xy,t\}z+x\{y,zt\}+y\{x,zt\} \\
&&-x\{zy,t\}-x\{y,zt\}-z\{yx,t\}-z\{x,yt\} -y\{tx,z\}-y\{x,tz\}-t\{xy,z\}-t\{y,xz\} \\
&=&-x\{zy,t\}-z\{x,yt\}-y\{tx,z\}-t\{y,xz\} \\
&=&-\{xzy,t\}-\{zy,xt\}-\{zx,yt\}-\{x,zyt\}-\{ytx,z\}-\{tx,yz\}-\{ty,xz\}-%
\{y,txz\} \\
&=&-\{xzy,t\}-\{x,zyt\}-\{ytx,z\}-\{y,txz\} \\
&=&-zy\{x,t\}-tx\{y,z\}.\end{array}$}\]%
So we obtain%
\[\adjustbox{max width=\linewidth}{$\displaystyle \{xy,zt\}+zy\{x,t\}+tx\{y,z\} \ =\ 0.$}\]%
Therefore%
\[\adjustbox{max width=\linewidth}{$\displaystyle \{xy,zt\}+yz\{x,t\}+xt\{y,z\}=\{xy,z\}t+\{xy,t\}z+x\{y,zt\}+y\{x,zt\}-xz%
\{y,t\}-yt\{x,z\}=0.$}\]%
Hence%
\[\adjustbox{max width=\linewidth}{$\begin{array}{rcl}0 &=&\{xy,z\}t+\{xy,t\}z+x\{y,zt\}+y\{x,zt\}-xz\{y,t\}-yt\{x,z\} \\
&=&\{xy,z\}t+\{xy,t\}z+x\{y,zt\}+y\{x,zt\}-x\{zy,t\}-x\{y,zt\}-y\{tx,z\}-y%
\{x,tz\} \\
&=&\{xy,z\}t+\{xy,t\}z-x\{zy,t\}-y\{tx,z\} \\
&=&\{xyt,z\}+2\{xy,zt\}-\{zy,xt\}-\{ytx,z\}-\{tx,yz\} \\
&=&2\{xy,zt\}.\end{array}$}\]%

Thus $\{xy,zt\}=0$, as required.
 \end{proof}

\begin{example}
Consider the following $5$-dimensional Poisson-type algebra$:$%
\[\adjustbox{max width=\linewidth}{$\displaystyle \rm{P}:=\left\{ 
\begin{tabular}{llll}
$e_{1}\cdot e_{1}=e_{3},$ & $e_{1}\cdot e_{2}=e_{5},$ & $e_{1}\cdot
e_{3}=e_{4},$ &  \\ 
$\{e_{2},e_{1}\}=-e_{1},$ & $\{e_{2},e_{3}\}=-e_{3},$ & $\{e_{2},e_{4}%
\}=-e_{4},$ & $\{e_{2},e_{5}\}=-e_{5}.$%
\end{tabular}%
\right.$}\]%
$\rm{P}$ is a transposed $1$-Poisson algebra with the identity $%
\{xy,zt\}=0$, but $\rm{P}$ is not an $F$-manifold algebra.
Moreover, $\rm{P}$ is not an algebra of Jordan brackets.
The last observation will be useful in connection with the next proposition.
\end{example}
\begin{proposition}\label{transJB}
    Let $\delta\notin \big\{0,\frac{1}{2},1\big\}$ and let $({\rm P}, \cdot, \{\cdot,\cdot\})$ be a transposed $\delta$-Poisson algebra. Then 
$({\rm P}, \cdot, \{\cdot,\cdot\})$ is an algebra of Jordan brackets if and only if it satisfies 
         $\{xy,zt \} \ = \ \{xyz,t \} \ =\ 0.$
\end{proposition}
\begin{proof}

 If $({\rm P}, \cdot, \{\cdot,\cdot\})$ is  transposed $\delta$-Poisson, 
  then, due to \eqref{idtp3} and \eqref{idtp4}, it satisfies the first identity of Jordan brackets.
Now we take the second identity of Jordan brackets and, using \eqref{deltranspois} and \eqref{idtp5}, rewrite it in the following form:
\begin{center}
    $\delta\big(\{xyt,z\} +\{yt,zx\}\big) +\frac{\delta^2}{1-\delta}\big(\{xy,zt\}+\{xt,yz\}\big) \ = \ \delta\big(\{xyt,z\}+\{xt,yz\}\big)+\frac{\delta^2}{1-\delta}\big(\{xy,zt\}+\{yt,xz\}\big).$
\end{center}
Hence, $\{xt,yz\}=\{yt,xz\},$ which gives
\begin{center}
 $\{xt,yz\} \ =\ \{yt,xz\} \ = \ \{zy,xt\}\ =\ - \{xt,yz\}.$   
\end{center}
It follows that $\{xt,yz\} = 0.$ Now the third identity of Jordan brackets reduces to 
\begin{center}
    $\{xyz,t\} \ = \ \{ty,z\}x+\{tx,z\}y$
\end{center}
  and, rewriting the right-hand side of the obtained identity using \eqref{deltranspois}, we have
  $\{xyz,t\}=2\delta  \{ xyt,z\}.$
The last gives 
\begin{center}
      $\{xyz,t\}=2\delta  \{ xyt,z\}= 4\delta^2 \{xyz,t\}$ \  and \  
      $\{xyz,t\}=2\delta  \{ xzt,y\}= 4\delta^2 \{xyt,z\}=8\delta^3 \{xyz,t\}.$ 
\end{center}
Hence $\{xyz,t\}=0$, which gives our statement.
\end{proof}

\begin{theorem}\label{simpletrdelta}
    Let $\delta\notin \big\{ 0,1\big\}$ and  $({\rm P}, \cdot, \{\cdot,\cdot\})$ be 
    a   simple transposed $\delta$-Poisson algebra.
    Then $({\rm P},   \{\cdot,\cdot\})$ is simple.
\end{theorem}

\begin{proof}
 Our proof is based on the ideas used for transposed $\frac{1}{2}$-Poisson algebras in \cite{fer23}.
It is easy to see that in each transposed $\delta$-Poisson algebra $({\rm P}, \cdot, \{\cdot,\cdot\})$, 
$\{ {\rm P},{\rm P} \}$ is an ideal. Hence, $({\rm P}, \{ \cdot,\cdot\})$ is perfect,  
i.e. $\{{\rm P}, {\rm P}\}={\rm P}.$

Let $I$ be an ideal of $({\rm P}, \{\cdot, \cdot\}),$ then
by (\ref{idtp2}) we have
\begin{center}
$\{ {\rm P}I, \{{\rm P},{\rm P}\}\} \subseteq \{{\rm P}{\rm P}, \{{\rm P},I\}\}+
\{{\rm P}{\rm P}, \{I,{\rm P}\}\} \subseteq I$ and 
$\{ {\rm P}I,  {\rm P}\}  \subseteq I.$
\end{center}

 Now, let $I\neq{\rm P}$ and let $I'$ be the maximal subspace such that $\{{\rm P},I'\} \subseteq I.$
 Then $I' \neq {\rm P}$ (otherwise $\{{\rm P},{\rm P}\}\subseteq I$), $I\subseteq I'$, and ${\rm P}I \subseteq I'$ by the previous inclusion.
 By (\ref{idtp1}), we have 
 \begin{center}
 $I'\{{\rm P},{\rm P}\} \subseteq {\rm P}\{{\rm P},I'\}+{\rm P} \{I',{\rm P}\} \subseteq {\rm P}I\subseteq I'.$
 \end{center}
Hence, $I'{\rm P}=I'\{{\rm P},{\rm P}\} \subseteq I'$ and $\{{\rm P},I'\} \subseteq I \subseteq I'$
which gives that  $I'$ is an ideal of $({\rm P}, \cdot, \{\cdot, \cdot\})$ and  $I'=0.$ It follows that $({\rm P},   \{\cdot, \cdot\})$ is simple.
\end{proof}

\begin{theorem}\label{simpleclass}
    Let $\delta\notin \big\{ 0,1\big\}$ and  $({\rm P}, \cdot, \{\cdot,\cdot\})$ be 
    a nontrivial complex simple finite-dimensional transposed $\delta$-Poisson algebra.
    Then $\delta=-1$ and $({\rm P}, \cdot, \{\cdot,\cdot\})$ is isomorphic to 
    one of the following algebras

\begin{enumerate}[$\bullet$]
    \item $\rm{A}_{1}:\left\{ 
\begin{array}{lllllllll}
e_{1}\cdot e_{1}&=&e_{2}, \\ 
\left\{ e_{1},e_{2}\right\}&=&e_{3},&
\left\{ e_{1},e_{3}\right\}&=&e_{1},&
\left\{ e_{2},e_{3}\right\}&=&-e_{2}.%
\end{array}%
\right. $

 \item  $\rm{A}_{2}:\left\{ 
\begin{array}{lllllllll}
e_{1}\cdot e_{1}& =& e_{3}, & e_{1}\cdot e_{3}& =& -e_{2}, \\ 
\left\{ e_{1},e_{2}\right\}&=&e_{3},&
\left\{ e_{1},e_{3}\right\}&=&e_{1},&
\left\{ e_{2},e_{3}\right\}&=&-e_{2}.%
\end{array}%
\right. $
 \end{enumerate}
\end{theorem} 
\begin{proof}
Let  $({\rm P}, \cdot, \{\cdot,\cdot\})$ be 
    a complex simple finite-dimensional transposed $\delta$-Poisson algebra.
Thanks to Theorem \ref{simpletrdelta}, $({\rm P},   \{\cdot,\cdot\})$ is simple.
Thanks to Proposition \ref{tpstr}, 
if $({\rm P},   \{\cdot,\cdot\}) \not\cong {\rm sl}_2$ or $\delta\neq-1,$ then 
$({\rm P},   \{\cdot,\cdot\})$ does not admit nontrivial transposed $\delta$-Poisson structures.
The case $\delta=-1$ and  $({\rm P},   \{\cdot,\cdot\}) \cong {\rm sl}_2$ 
is treated by a direct calculation, which gives, up to isomorphism, only two nontrivial possibilities.
Namely, $({\rm P}, \cdot, \{\cdot,\cdot\})$ is isomorphic to one of algebras from our statement.
\end{proof}

\begin{remark}
It is known that the Kantor double of simple   Poisson and transposed Poisson algebras is also simple. 
Thanks to Proposition \ref{transJB},  $\rm{A}_{1}$ is Jordan brackets, 
hence the Kantor double $J(\rm{A}_{1})$ is a Jordan superalgebra. 
$\rm{A}_{1}$  gives an example of   a simple transposed $(-1)$-Poisson algebra with nilpotent Kantor double.
\end{remark}

\section{Poisson-type algebras in one multiplication}\label{onemu}
 
Let $({\mathcal A}, \cdot)$ be an algebra. We consider the following two new products
on the underlying vector space ${\mathcal A}$ defined by 
\begin{center}
$x\circ y :=\frac{1}{2}(x\cdot y+y\cdot x), \hspace{2cm} [x,y] :=\frac{1}{2}%
(x\cdot y-y\cdot x).$
\end{center}
Every generic Poisson--Jordan  algebra 
(resp., generic Poisson algebra) $({\mathcal A},\circ, [\cdot,\cdot])$ is
associated with precisely one noncommutative Jordan (resp., Kokoris) algebra $({\mathcal A}%
,\cdot)$. That is, we
have a bijective correspondence between generic Poisson--Jordan (resp., generic Poisson) algebras
  and
noncommutative Jordan (resp., Kokoris) algebras \cite{aak}.
 The main aim of the present section is to construct a one-to-one correspondence between the new Poisson-type algebras and suitable varieties of algebras with one multiplication.

\subsection{$\delta$-Poisson algebras in one multiplication}
One of the first papers devoted to Poisson algebras in one multiplication is due to Goze and Remm \cite{GR}.
The present subsection is devoted to the extension of some results of Goze and Remm to the $\delta$-Poisson case.
\begin{proposition} \label{delta}
Consider an algebra $\left( \mathcal{A},\cdot \right) $ whose underlying
vector space $\mathcal{A}$ is endowed with the two new products $\circ $ and 
$\left[ \cdot ,\cdot \right] $. If $\delta \neq 0$, then $(\mathcal{A},\circ
,\left[ \cdot ,\cdot \right] )$ is a ${\delta }${-Poisson algebra if and
only if the }algebra $\left( \mathcal{A},\cdot \right) $ satisfies the
following identity$:$%
\begin{equation}
f_{\delta}(x,y,z)\ :=\ 3\delta \left( xy\right) z+\left( 1-2\delta \right) y\left(
zx\right) -\left( 2\delta +1\right) x\left( yz\right) -x\left( zy\right)
+y\left( xz\right) +\delta z\left( xy\right) \ =\ 0.
\label{id-delta-poisson}
\end{equation}%
\end{proposition}

\begin{proof}
Suppose first that $(\mathcal{A},\circ ,\left[ \cdot ,\cdot \right] )$ is a ${%
\delta }${-Poisson algebra. }Since $(\mathcal{A},\circ )$ is associative, we
have the identity $\left( x\circ y\right) \circ z=x\circ \left( y\circ
z\right) $ which is equivalent to:%
\begin{center}\adjustbox{max width=\linewidth}{\begin{tabular}{rclcl}
$h_{1}(x,y,z)$&$:=$&$(xy)z+(yx)z-(zy)x-(yz)x-x(yz)-x(zy)+z(yx)+z(xy)$&$=$&$0.$
\end{tabular}}\end{center}%
Also, since $(\mathcal{A},\left[ \cdot ,\cdot \right] )$ is a Lie algebra, the
Jacobi identity gives:%
\begin{center}\adjustbox{max width=\linewidth}{\begin{tabular}{rclcrcrc}
$h_{2}(x,y,z)$&$:=$&$
(xy)z-(yx)z-(zy)x-(xz)y+(yz)x+(zx)y$\\
&&$-x(yz)+y(xz)+z(yx)+x(zy)-y(zx)-z(xy)$&$=$&$0.$
\end{tabular}}\end{center}%
Finally, the identity \ 
$\left[ x\circ y,z\right] =\delta \big( x\circ \left[ y,z%
\right] +  \left[ x,z\right] \circ y \big)$ \  gives:%
\begin{center}\adjustbox{max width=\linewidth}{\begin{tabular}{rclcl}
$h_{3}^{\delta }(x,y,z)$&$:=$&$\left( xy\right) z+\left( yx\right) z-z\left(
xy\right) -z\left( yx\right) -\delta x\left( yz\right) +\delta x\left(zy\right)$ \\
&&$-\delta \left( yz\right) x+\delta \left( zy\right) x-\delta \left(
xz\right) y+\delta \left( zx\right) y-\delta y\left( xz\right) +\delta
y\left( zx\right) $&$=$&$0.$
\end{tabular}}\end{center}
Now it is straightforward to verify that%
\begin{center}\adjustbox{max width=\linewidth}{\begin{tabular}{rclcl}
$f_{\delta }(x,y,z)$ &$:=$&$\frac{1}{2}\left( \delta +1\right) h_{1}(x,y,z)+\frac{%
1}{2}\left( \delta -1\right) h_{1}(y,x,z)+\frac{1}{2}\delta h_{2}(x,y,z)$ \\
&&$+\frac{1}{2}\delta h_{3}^{\delta }(x,y,z)+\frac{1}{2}\left( \delta
-1\right) h_{3}^{\delta }(x,z,y)+\frac{1}{2}\left( \delta +1\right)
h_{3}^{\delta }(y,z,x)$&$=$&$0.$
\end{tabular}}\end{center}%
On the other hand, suppose that $(\mathcal{A},\cdot )$ satisfies \eqref{id-delta-poisson}. Then
\begin{center}\adjustbox{max width=\linewidth}{\begin{tabular}{rclll}
$ (x,  y, z)_{\circ} $ &$=$
&\multicolumn{2}{l}{
$\frac{1}{4\delta }\big(\delta (xy)z+\delta (yx)z-\delta (zy)x-\delta (yz)x -\delta x(yz)-\delta x(zy)+\delta z(yx)+\delta z(xy) \big)$} \\
&$=$&$\frac{1}{12\delta} \Big( $&$f_{\delta }(x,y,z)-\left( 1-2\delta \right) y\left( zx\right) +\left(
2\delta +1\right) x\left( yz\right) +x\left( zy\right) -y\left( xz\right)
-\delta z\left( xy\right)$  \\
&&&$+f_{\delta }(y,x,z)-\left( 1-2\delta \right) x\left( zy\right) +\left(
2\delta +1\right) y\left( xz\right) +y\left( zx\right) -x\left( yz\right)
-\delta z\left( yx\right)$  \\
&&&$-f_{\delta }(z,y,x)+\left( 1-2\delta \right) y\left( xz\right) -\left(
2\delta +1\right) z\left( yx\right) -z\left( xy\right) +y\left( zx\right)
+\delta x\left( zy\right)$  \\
&&&$-f_{\delta }(y,z,x)+\left( 1-2\delta \right) z\left( xy\right) -\left(
2\delta +1\right) y\left( zx\right) -y\left( xz\right) +z\left( yx\right)
+\delta x\left( yz\right)  $  \\
\multicolumn{4}{r}{$-3\delta x(yz)-3\delta x(zy)+3\delta z(yx)+3\delta z(xy) \Big)$} \\
&$=$&\multicolumn{2}{l}{$\frac{1}{12\delta } \Big(f_{\delta }(x,y,z)+f_{\delta }(y,x,z)-f_{\delta }(z,y,x)-f_{\delta
}(y,z,x) \Big) \ = \ 0.$}
\end{tabular}}\end{center}
Hence $x\circ \left( y\circ z\right) =\left( x\circ y\right) \circ z$ and therefore $(\mathcal{A},\circ )$ is associative. Similarly, by a straightforward calculation we have
 
\begin{center}\adjustbox{max width=\linewidth}{\begin{tabular}{lclcl}
$\left[ \left[ x,y\right] ,z\right] +\left[ \left[ y,z\right] ,x\right] +
\left[ \left[ z,x\right] ,y\right]  $&$=$&$\frac{1}{3\delta }
\big( 
f_{\delta}(x,y,z)-f_{\delta }(x,z,y)-f_{\delta }(y,x,z)+$\\
\multicolumn{3}{r}{$f_{\delta }(z,x,y)+f_{\delta}(y,z,x)-f_{\delta }(z,y,x)\big)$} &$=$&$0,$ \\
$\left[ x\circ y,z\right] -\delta\big( x\circ \left[ y,z\right] + \left[ x,z\right] \circ y \big)$&$=$&$\frac{1}{3\delta }\big( 
f_{\delta }(x,y,z)+f_{\delta}(y,x,z)\big) +\frac{1}{3}\big( f_{\delta}(z,x,y)-$\\
\multicolumn{3}{r}{$f_{\delta }(x,z,y)-
f_{\delta }(y,z,x)+f_{\delta }(z,y,x)\big)$} &$=$&$0.$
\end{tabular}}\end{center}
 \end{proof}

\begin{proposition}\label{37pro}
    Let $\left( \mathcal{A},\cdot \right) $ be an algebra. If $(\mathcal{A},\cdot)$ 
satisfies the identity \eqref{id-delta-poisson}  with $\delta \neq 0$, then $(\mathcal{A},\cdot)$  is power-associative.
\end{proposition}

\begin{proof}
We have $\left( x,x,x\right) =\frac{1}{3\delta }f_{\delta }(x,x,x)=0$. This implies 
\begin{center}$\left( x,x,x^{2}\right) \  =\ \frac{1}{4\delta }f_{\delta
}(x,x,x^{2}) \ =\ 0$.\end{center} 
Then the opposite algebra $\mathcal{A}^{op}$ satisfies the conditions of \cite[Lemma 3]{Albert}; hence, $\mathcal{A}^{op}$ is power-associative. Therefore, $(\mathcal{A},\cdot)$ \ is power-associative.
\end{proof}

\begin{corollary}
Let $(\mathcal{A},\cdot)$  be a finite dimensional algebra satisfying \eqref{id-delta-poisson}  with $\delta \neq 0$. If $ \mathcal{A}$ is not a nilalgebra, then $ \mathcal{A}$ contains a non-zero idempotent element.

\end{corollary}

\begin{proposition}\label{unitalP}
Let $\left( \mathcal{A},\cdot \right) $ be an algebra satisfying \eqref{id-delta-poisson} with $\delta \neq 0$. Then 
 
\begin{enumerate}
    \item[{\rm 1.}] if $\left( \mathcal{A},\cdot \right) $ is commutative, then $(\mathcal{A},\cdot)$ is associative$;$
\item[{\rm 2.}]   if $\delta\neq1$ and $\left( \mathcal{A},\cdot \right) $ has an identity element, then it is associative and commutative. 

\end{enumerate}
\end{proposition}

\begin{proof}
If $\left( \mathcal{A},\cdot \right)$ is commutative, then the identity \eqref{id-delta-poisson} gives 
\begin{center}
    $2\delta (xy)z +(1-\delta) (xz)y-(\delta+1) (yz)x \ = \ 0,$
\end{center}
which is equivalent to
\begin{center}
    $2\delta (xy)z +(1-\delta) (yz)x-(\delta+1) (xz)y \ = \ 0.$
\end{center}
Adding the last two identities, we have 
$2(xy)z=(xz)y+x(yz),$ which gives \begin{center}
$4(xy)z=2(xz)y+2x(yz)=(xy)z+x(zy)+2x(yz),$ 
 i.e. $(xy)z=x(yz)$.\end{center} 

For the second part of the statement, suppose that $\left( \mathcal{A},\cdot \right)$ has an identity element and take $z=1$ in \eqref{id-delta-poisson}. This gives $2(\delta-1)(xy-yx)=0$; hence, $\left( \mathcal{A},\cdot \right)$ is commutative and, by the first part, associative.
For $\delta=1$, the second part does not hold: if $(\mathcal{A},\circ,[\cdot,\cdot])$ is the unital Poisson algebra with basis $1,p,q,pq$, where $p^2=q^2=0$ and $\{p,q\}=pq$, then the algebra $(\mathcal{A},\cdot)$ with $xy=x\circ y+[x,y]$ satisfies \eqref{id-delta-poisson} for $\delta=1$, has the identity element $1$, and is not commutative. \end{proof}

\begin{definition}
 An algebra $(\mathcal{A}, \cdot)$ is
called flexible if the following identity holds
\begin{eqnarray}\label{flexible law}
    {\rm FL} (x,y,z) \ := \ \left( x,y,z\right) +  \left( z,y,x\right) \ = \ 0. 
\end{eqnarray}
 An algebra $(\mathcal{A}, \cdot)$ is
called anti-flexible if the following identity holds
\[\adjustbox{max width=\linewidth}{$\begin{array}{rcl}{\rm AFL} (x,y,z) \ := \ \left( x,y,z\right) -  \left( z,y,x\right) \ = \ 0.\end{array}$}\] 
\end{definition}

\begin{proposition}\label{ass-shift}
Let $(\mathcal{A}, \cdot)$ be an algebra satisfying \eqref{id-delta-poisson} with $\delta \notin \big\{ 0,\pm 1 \big\}$%
. Then $(\mathcal{A}, \cdot)$ is associative if and only if $(\mathcal{A},\cdot)$ is a shift
associative algebra.
\end{proposition}

\begin{proof}
The variety of shift associative algebras is defined by the identity
\[\adjustbox{max width=\linewidth}{$\displaystyle S\left( x,y,z\right) \ :=\ \left( xy\right) z-y\left( zx\right) \ =\ 0.$}\]%
Then it is straightforward to verify that
\begin{center}\adjustbox{max width=\linewidth}{\begin{tabular}{lcl}
$S(x,y,z)$ &$=$&$\frac{\delta -1}{3\delta \left( \delta -1\right) }f_{\delta
}(x,y,z)+\frac{1}{3\delta \left( \delta -1\right) }f_{\delta }(x,z,y)-\frac{%
\delta }{3\delta \left( \delta -1\right) }f_{\delta }(z,x,y)+\frac{1}{3\delta \left( \delta -1\right) }f_{\delta }(z,y,x)$ \\
\multicolumn{3}{r}{$+\frac{1}{%
1-\delta }(x,z,y)+\frac{\delta }{\delta -1}(z,x,y)+\frac{1}{1-\delta }(z,y,x),$}
\\
$(x,y,z) $&$=$&$\frac{1+2\delta }{3\delta \left( \delta +1\right) }f_{\delta
}(x,y,z)+\frac{1}{3\delta \left( \delta +1\right) }f_{\delta }(y,x,z)+\frac{%
\delta }{3\delta \left( \delta +1\right) }f_{\delta }(z,x,y)$ \\
 \multicolumn{3}{r}{$-\frac{\delta }{\delta +1}S(x,y,z)-\frac{1}{1+\delta }S\left( y,x,z\right)
-\frac{\delta }{1+\delta }S\left( z,x,y\right).$}
\end{tabular}}\end{center}
\end{proof}

\begin{corollary}
Let $(\mathcal{A}, \cdot)$ be an algebra satisfying \eqref{id-delta-poisson} with $\delta \notin \big\{ 0,\pm 1\big\}$%
. Then $(\mathcal{A}, \cdot)$ is associative if and only if $(\mathcal{A},  \cdot)$ is a cyclic
associative algebra. 
\end{corollary}

\begin{proposition}
Let $(\mathcal{A}, \cdot)$ be an algebra satisfying \eqref{id-delta-poisson} with $\delta =1$. If $%
(\mathcal{A}, \cdot)$ is a shift associative algebra then it is associative and
therefore cyclic associative.
\end{proposition}

\begin{proof}
It is straightforward to verify that:%
\begin{center}\adjustbox{max width=\linewidth}{\begin{tabular}{lcl}
$(x,y,z)$ &$=$&$\frac{1}{2}f_{1}(x,y,z)+\frac{1}{6}f_{1}(y,x,z)+\frac{1}{6}%
f_{1}(z,x,y)-\frac{1}{2}S(x,y,z)-\frac{1}{2}S\left( y,x,z\right) -\frac{1}{2}%
S\left( z,x,y\right).$
\end{tabular}}\end{center}
\end{proof}

\begin{proposition}
Let $(\mathcal{A}, \cdot)$ be an algebra satisfying \eqref{id-delta-poisson} with $\delta =-1$. 
Then $(\mathcal{A}, \cdot)$ is a shift associative algebra if and only if it is
anti-flexible.
\end{proposition}

\begin{proof}
It is straightforward to verify that
\begin{center}\adjustbox{max width=\linewidth}{\begin{tabular}{lcl}
$S\left( x,y,z\right)  $&$=$&$-\frac{1}{2}f_{-1}(x,y,z)+\frac{1}{6}f_{-1}(x,z,y)-%
\frac{1}{6}f_{-1}(y,z,x)+\frac{1}{6}f_{-1}(z,y,x)-$\\
\multicolumn{3}{r}{$\frac{1}{2} {\rm AFL}(x,y,z)+\frac{1}{2} {\rm AFL}\left(
x,z,y\right),$} \\
$ {\rm AFL}(x,y,z)$&$=$&$-\frac{1}{3}f_{-1 }(x,y,z)-\frac{2}{3}f_{-1 }(y,x,z)+%
\frac{1}{3}f_{-1 }(z,y,x)-2S\left( y,x,z\right).$
\end{tabular}}\end{center}
\end{proof}

\begin{corollary}
Let $(\mathcal{A},\cdot)$ be an algebra satisfying \eqref{id-delta-poisson} with $\delta =-1$. If $(\mathcal{A},\cdot)$  is associative, then $(\mathcal{A},\cdot)$ is a shift associative
algebra and therefore a cyclic associative.
\end{corollary}

 Let us remember that each algebra satisfies the identity \eqref{id-delta-poisson} for $\delta=1$ is flexible \cite[Proposition 4]{GR}. 
The next statement gives a different situation for an arbitrary $\delta.$

\begin{proposition}
If the variety  of algebras defined by the identity \eqref{id-delta-poisson} with $\delta \neq 0$
is a subvariety of flexible algebras, then $\delta=1.$
\end{proposition}

\begin{proof}
If the identity (\ref{flexible law})  is a consequence of \eqref{id-delta-poisson}, then there exist $\alpha
_{1},\ldots ,\alpha _{6}$ such that: 
\begin{flushleft}
$ {\rm FL} (x,y,z) =$
\end{flushleft}
\begin{flushright}$\alpha _{1}f_{\delta
}(x,y,z)+\alpha _{2}f_{\delta }(x,z,y)+\alpha _{3}f_{\delta }(y,x,z)+\alpha
_{4}f_{\delta }(z,x,y)+\alpha _{5}f_{\delta }(y,z,x)+\alpha _{6}f_{\delta
}(z,y,x).$
\end{flushright}
Then $\alpha _{1}=\alpha _{6}=\frac{1}{3},\alpha _{2}=\alpha _{3}=\alpha
_{4}=\alpha _{5}=0$ and $\delta =1$.
\end{proof}

\begin{corollary}
Let $\left( \mathcal{A},\cdot \right) $ be an algebra. If $(\mathcal{A},\cdot)$
satisfies the identity \eqref{id-delta-poisson} with $\delta = 1,$ then $(\mathcal{A},\cdot)$ is a Kokoris algebra and therefore a noncommutative Jordan algebra.    
\end{corollary}

\begin{definition}
An algebra $(\mathrm{P},\cdot ,\{\cdot,\cdot \})$ is called a generic $\delta$-Poisson (resp., nonassociative $\delta$-Poisson) algebra if $(\mathrm{P},\cdot )$ is a commutative
associative (resp., commutative) algebra, $(\mathrm{P},\{\cdot,\cdot \})$ is anticommutative (resp., Lie) and the identity \eqref{deltpois} holds. 
\end{definition}
 
It is clear that the variety of $\delta$-Poisson algebras is the intersection of the variety of generic $\delta$-Poisson algebras and the variety of nonassociative $\delta$-Poisson algebras.

\begin{proposition}
Consider an algebra $\left( \mathcal{A},\cdot \right) $ whose underlying
vector space $\mathcal{A}$ is endowed with the two new products $\circ $ and 
$\left[ \cdot ,\cdot \right] $. 
Then $(\mathcal{A},\circ ,
\left[ \cdot ,\cdot \right] )$ is a $0$-Poisson algebra if and only
if the algebra $\left( \mathcal{A},\cdot \right) $ satisfies the following
identities$:$%
\begin{eqnarray}
g_{1}\left( x,y,z\right)  &:=& (xy)z+(zx)y-(zy)x+x(zy)-z(xy)-y(zx)\ =\ 0,
\label{id1} \\
g_{2}\left( x,y,z\right)  &:=& (xz)y+(zx)y-z(xy)-z(yx)\ =\ 0.  \label{id2}
\end{eqnarray}
\noindent
Moreover, 
$(\mathcal{A},\circ ,
\left[ \cdot ,\cdot \right] )$ is a nonassociative $0$-Poisson algebra if and only
if the algebra $\left( \mathcal{A},\cdot \right) $ satisfies $\eqref{id1}$, $(\mathcal{A},\circ ,
\left[ \cdot ,\cdot \right] )$ is a generic $0$-Poisson algebra if and only
if the algebra $\left( \mathcal{A},\cdot \right) $ satisfies $\eqref{id2}$. 

\end{proposition}

\begin{proof}
 
As in Proposition \ref{delta}, the identity $\left( x\circ
y\right) \circ z=x\circ \left( y\circ z\right) ${\ is equivalent to}%
\[\adjustbox{max width=\linewidth}{$\displaystyle h_{1}\left( x,y,z\right) \ :=\ (xy)z+(yx)z-(zy)x-(yz)x-x(yz)-x(zy)+z(yx)+z(xy) \ =\ 0,$}\]%
while the Jacobi identity is equivalent to:%
\begin{center}\adjustbox{max width=\linewidth}{\begin{tabular}{rclcrcrc}
$h_{2}(x,y,z)$&$:=$&$
(xy)z-(yx)z-(zy)x-(xz)y+(yz)x+(zx)y$\\
&&$-x(yz)+y(xz)+z(yx)+x(zy)-y(zx)-z(xy)$&$=$&$0.$
\end{tabular}}\end{center}\noindent
Moreover, the identity $[x \circ y,z]=0$ gives:%
\[\adjustbox{max width=\linewidth}{$\displaystyle h_{3}^{0}(x,y,z)\ :=\ \left( xy\right) z+\left( yx\right) z-z\left( xy\right) -z\left(
yx\right) \ =\ 0.$}\]%
Now it is straightforward to verify that

    \begin{center}\adjustbox{max width=\linewidth}{\begin{tabular}{lcl}
$g_{1}\left( x,y,z\right)$&$ =$&$\frac{1}{2}h_{2}(x,y,z)+\frac{1}{2}h_{3}^{0}(x,y,z)+\frac{1}{2}h_{3}^{0}(x,z,y)-\frac{1}{2}h_{3}^{0}(y,z,x),$\\

$g_{2}\left( x,y,z\right)$&$ =$&$-\frac{1}{2}h_{1}(y,x,z)+\frac{1}{2}%
h_{3}^{0}(x,y,z)+\frac{1}{2}h_{3}^{0}(x,z,y),$\\

$h_{1}\left( x,y,z\right)$ &$=$&$g_{2}\left( y,z,x\right) -g_{2}\left(
y,x,z\right)$, \\
$h_{3}^{0}(x,y,z) $&$=$&$g_{2}\left( x,z,y\right)+g_{2}\left( y,x,z\right)-g_{2}\left( z,x,y\right),$\\

$h_{2}\left( x,y,z\right)  $&$=$&$g_{1}\left( x,y,z\right) -g_{1}\left(
x,z,y\right)$,\\
$h_{3}^{0}\left( x,y,z\right)  $&$=$&$g_{1}\left( x,y,z\right) +g_{1}\left(
y,x,z\right)$.
\end{tabular}}\end{center}

\noindent It follows that $(\mathcal{A},\circ ,
\left[ \cdot ,\cdot \right] )$ is a nonassociative $0$-Poisson algebra if and only
if the algebra $\left( \mathcal{A},\cdot \right) $ satisfies ($\ref{id1}$), $(\mathcal{A},\circ ,
\left[ \cdot ,\cdot \right] )$ is a generic $0$-Poisson algebra if and only
if the algebra $\left( \mathcal{A},\cdot \right) $ satisfies ($\ref{id2}$) and $(\mathcal{A},\circ ,
\left[ \cdot ,\cdot \right] )$ is a $0$-Poisson algebra if and only
if the algebra $\left( \mathcal{A},\cdot \right) $ satisfies ($\ref{id1}$) and ($\ref{id2}$).

Now we consider the following algebras:
\begin{center}\adjustbox{max width=\linewidth}{\begin{tabular}{llllllllllll}
$\mathcal{B}_1$&$:$&
$e_{1}\cdot e_{1}=e_{2}\cdot e_{3}=e_{3}\cdot e_{2}=e_{4},$&$
e_{1}\cdot e_{2}=-e_{2}\cdot e_{1}=e_{2},$&$
e_{1}\cdot e_{3}=-e_{3}\cdot e_{1}=e_{1}.$\\
$\mathcal{B}_2$&$:$&$
e_{1}\cdot e_{1}= e_{2}\cdot e_{2}= e_{3}\cdot
e_{3}=e_{3},$&$
e_{1}\cdot e_{2}=-e_{2}\cdot e_{1}=e_{3}.$
\end{tabular}}\end{center}%
\noindent  $\mathcal{B}_1$ satisfies ($\ref{id2}$) but it does not satisfy  ($\ref{id1}$). $\mathcal{B}_2$ satisfies ($\ref{id1}$) but it does not satisfy  ($\ref{id2}$).
Hence, ($\ref{id1}$) and   ($\ref{id2}$)
are independent.
\end{proof}
The next statement can be proved similarly to Proposition \ref{37pro}.
\begin{corollary}
If $(\mathcal{A},\cdot)$
satisfies the identity   \eqref{id2}, then $(\mathcal{A},\cdot)$ is a power-associative algebra.    
\end{corollary}

\subsection{Scalar-Poisson algebras in one multiplication}
\begin{proposition}\label{sc}
  Consider an algebra $\left( \mathcal{A},\cdot \right) $ whose underlying
vector space $\mathcal{A}$ is endowed with the two new products $\circ $ and 
$\left[ \cdot ,\cdot \right] $. Then $(\mathcal{A},\circ ,\left[ \cdot
,\cdot \right] )$ is a {scalar-Poisson algebra if and only if the }algebra $%
\left( \mathcal{A},\cdot \right) $ satisfies the following identities$:$%
\begin{eqnarray}
A_{1}\left( x,y,z\right) \ &:=& \ 3\left( xz\right) y-2x\left( zy\right) +z\left(
xy\right) -y\left( zx\right) -z\left( yx\right) \ =\ 0,  \label{sc1} \\
A_{2}\left( x,y,z\right)  \ &:=& \ x(yz)+x(zy)-z(xy)-z(yx)\ = \ 0.  \label{sc2}
\end{eqnarray}

\end{proposition}

\begin{proof}
Suppose first that $(\mathcal{A},\circ ,\left[ \cdot ,\cdot \right] )$ is  a {%
scalar-Poisson algebra. As in Proposition \ref{delta}, the identity }$\left( x\circ
y\right) \circ z=x\circ \left( y\circ z\right) ${\ is equivalent to}%
\[\adjustbox{max width=\linewidth}{$\displaystyle h_{1}\left( x,y,z\right) \ :=\ (xy)z+(yx)z-(zy)x-(yz)x-x(yz)-x(zy)+z(yx)+z(xy) \ =\ 0,$}\]%
while the Jacobi identity is equivalent to:%
\begin{center}\adjustbox{max width=\linewidth}{\begin{tabular}{rclcrcrc}
$h_{2}(x,y,z)$&$:=$&$
(xy)z-(yx)z-(zy)x-(xz)y+(yz)x+(zx)y$\\
&&$-x(yz)+y(xz)+z(yx)+x(zy)-y(zx)-z(xy)$&$=$&$0.$
\end{tabular}}\end{center}\noindent 
Moreover, the identity $[x \circ y,z]=0$ gives:%
\[\adjustbox{max width=\linewidth}{$\displaystyle h_{3}^{0}(x,y,z)\ :=\ \left( xy\right) z+\left( yx\right) z-z\left( xy\right) -z\left(
yx\right) \ =\ 0$}\]%
and the {identity }$x \circ [y,z]+[x,z]\circ y=0$ gives:%
\begin{center}\adjustbox{max width=\linewidth}{\begin{tabular}{lclclclc}
$E(x,y,z)$&$:=$&$x\left( yz\right) -x\left( zy\right) +\left( yz\right) x-\left(
zy\right) x+\left( xz\right) y-\left( zx\right) y+y\left( xz\right) -y\left(
zx\right) $&$=$&$0.$
\end{tabular}}\end{center}%
Now it is straightforward to verify that:%
\begin{center}\adjustbox{max width=\linewidth}{\begin{tabular}{lclcl}
$A_{1}\left( x,y,z\right)  $&$=$&$\frac{1}{2}h_{1}(x,z,y)-\frac{1}{2}%
h_{2}(x,y,z)+h_{3}^{0}(x,z,y)+$\\
\multicolumn{3}{r}{$\frac{1}{2}h_{3}^{0}(y,z,x)-E(x,y,z)-\frac{1}{2}E(x,z,y)$}&$ = $&$0,$ \\
$A_{2}\left( x,y,z\right)  $&$=$&$-\frac{1}{2}h_{1}(x,y,z)+\frac{1}{2}h_{3}^{0}(x,y,z)-\frac{1%
}{2}h_{3}^{0}(y,z,x)$&$ =$&$ 0.$
\end{tabular}}\end{center}%
Conversely, let $(\mathcal{A},\cdot)$ satisfies the identities \eqref{sc1} and ($%
\ref{sc2}$).  Then 
\begin{center}\adjustbox{max width=\linewidth}{\begin{tabular}{lclcl}
$h_{1}(x,y,z)$ &$=$&$\frac{1}{3} \big(A_{1}(x,z,y)-A_{1}(y,x,z)-A_{1}(z,x,y)+ A_{1}(y,z,x) \big)-A_{2}(x,y,z)$&$=$&$0,$ \\
$h_{2}(x,y,z) $&$=$&$
\frac{1}{3}\big(A_{1}(x,z,y)-A_{1}(x,y,z)+
A_{1}(y,x,z)-$\\
\multicolumn{3}{r}{$A_{1}(z,x,y)-A_{1}(y,z,x)+A_{1}(z,y,x) \big)$}&$=$&$0,$ \\
$h_{3}^{0}(x,y,z) $&$=$&$\frac{1}{3} \big(A_{1}(x,z,y)+A_{1}(y,z,x)+2A_{2}(x,y,z)-A_{2}(x,z,y)\big)$&$=$&$0,$ \\
$E(x,y,z)$ &$=$&$\frac{1}{3} \big(A_{1}(x,y,z)+A_{1}(y,x,z)-
A_{1}(z,x,y)-$\\
\multicolumn{3}{r}{$A_{1}(z,y,x)+2A_{2}(x,y,z)-
A_{2}(x,z,y) \big)$}&$=$&$0.$
\end{tabular}}\end{center}

Now we consider the following algebras:
\begin{center}\adjustbox{max width=\linewidth}{\begin{tabular}{llllllllllllllll}
$\mathcal{B}_1$ &$:$&$ e_{1}\cdot e_{1}=e_{1},$&
$e_{2}\cdot e_{2}=e_{2},$&
$e_{1}\cdot e_{3}=   \frac{1}{3}
e_{3},$&
$e_{3}\cdot e_{1}= \frac{2}{3}  e_{3}.$\\

$\mathcal{B}_2$ &$:$&$ e_{1}\cdot e_{1}=e_{1},$&
$e_{2}\cdot e_{2}=e_{2},$&
$e_{1}\cdot e_{3}= \frac{2}{3} 
e_{3},$&
$e_{3}\cdot e_{1}=  \frac{1}{3} e_{3}.$\\
\end{tabular}}\end{center}%
\noindent
$\mathcal{B}_1$ satisfies (\ref{sc1}) but it
does not satisfy (\ref{sc2}).  
$\mathcal{B}_2$
satisfies (\ref{sc2}) but it does not satisfy (\ref{sc1}). 
Therefore, the identities  (\ref{sc1}) and (\ref{sc2})  are independent.\end{proof}

\begin{proposition}
    Let $\left( \mathcal{A},\cdot \right) $ be an algebra satisfying \eqref{sc1}. Then $(\mathcal{A},\cdot)$ is a flexible Lie-admissible and Jordan-admissible algebra.
    In particular, $(\mathcal{A},\cdot)$ is a noncommutative Jordan algebra and therefore power-associative.
\end{proposition}
\begin{proof}
Setting $y=x$ in \eqref{sc1}, we get $(x,z,x)=0$ and so $(\mathcal{A},\cdot)$ is flexible. Also, from the proof of Proposition \ref{sc}, we have  
\begin{center}\adjustbox{max width=\linewidth}{\begin{tabular}{lclcl}
$\left[ \left[ x,y\right] ,z\right] +\left[ \left[ y,z\right] ,x\right] +
\left[ \left[ z,x\right] ,y\right]$ &$=$&
$\frac{1}{3} \big(A_{1}(x,z,y)- A_{1}(x,y,z)+
A_{1}(y,x,z)-$ \\
\multicolumn{3}{r}{$ A_{1}(z,x,y)- A_{1}(y,z,x)+ A_{1}(z,y,x) \big)$}&$=$&$0.$
\end{tabular}}\end{center}
So $(\mathcal{A},\cdot)$ is Lie admissible.  It remains to show that $(\mathcal{A},\cdot)$ is Jordan admissible. To do this, expand each of the following:
\begin{center}\adjustbox{max width=\linewidth}{\begin{tabular}{rcl}
$\frac{1}{3}A_{1}(x,xx,y)+\frac{1}{3}A_{1}(x,y,xx)-\frac{1}{3}A_{1}(xx,y,x)-
\frac{1}{3}A_{1}(xx,x,y) $&$=$&$0,$ \\
$2A_{1}(y,x,x)x-2A_{1}(yx,x,x)-A_{1}(y,xx,x)+A_{1}(y,x,xx) $&$=$&$0,$ \\
$\frac{4}{3}xA_{1}(x,x,y) $&$=$&$0.$\\
\multicolumn{3}{l}{\mbox{Then we get}}\\
$\left( xx\right) \left( xy\right) +\left( xy\right) \left( xx\right) -\left(
\left( xx\right) y\right) x-x\left( \left( xx\right) y\right)  $&$=$&$0,$ \\
$\left( xx\right) \left( yx\right) +\left( yx\right) \left( xx\right) -\left(
y\left( xx\right) \right) x-4\left( x\left( xy\right) \right) x+4x\left(
x\left( yx\right) \right) -x\left( y\left( xx\right) \right)  $&$=$&$0,$ \\
$4x\left( \left( xy\right) x\right) -4x\left( x\left( yx\right) \right)  $&$=$&$0.$
\end{tabular}}\end{center}%
Taking the sum of the three identities above gives
\[\adjustbox{max width=\linewidth}{$\displaystyle \left( xx\right) \left( xy\right) +\left( xx\right) \left( yx\right) +\left(
xy\right) \left( xx\right) +\left( yx\right) \left( xx\right) =x\left(
\left( xx\right) y\right) +x\left( y\left( xx\right) \right) +\left( \left(
xx\right) y\right) x+\left( y\left( xx\right) \right) x.$}\]%
This   implies that $\left( x\circ x\right) \circ \left( x\circ y\right)
=x\circ \big( \left( x\circ x\right) \circ y\big) $ and therefore $(\mathcal{A},\cdot)$ is Jordan admissible.
\end{proof}

\begin{corollary}
Let $\left( \mathcal{A},\cdot \right) $ be an algebra satisfying \eqref{sc1} or \eqref{sc2}. Then 
\begin{enumerate}
    \item[{\rm 1.}] if $\left( \mathcal{A},\cdot \right)$
is commutative, then $(\mathcal{A},\cdot)$ is associative$;$
\item[{\rm 2.}]  if $\left( \mathcal{A},\cdot \right)$ is
unital, then $(\mathcal{A},\cdot)$ is commutative and associative.

\end{enumerate}
\end{corollary}

\begin{proposition}
 Let $( \mathcal{A},\cdot )$ satisfy the
identities $\eqref{sc1}$ and $\eqref{sc2}$. Then 
\begin{enumerate}
    \item[{\rm 1.}]  $( \mathcal{A},\cdot )$ is a Kokoris
algebra$;$   
\item[{\rm 2.}] if ${\rm Ann}(\mathcal{A}, [\cdot,\cdot]) = 0$, then $( \mathcal{A},\cdot )$ has no nonzero idempotents and if $( \mathcal{A},\cdot )$ is finite-dimensional, then it is a nilalgebra.
\end{enumerate}
\end{proposition}

\begin{proof}

 Let $( \mathcal{A},\cdot )$ satisfy the
identities \eqref{sc1} and \eqref{sc2}. Then $(\mathcal{A},\circ ,\left[ \cdot
,\cdot \right] )$ is a scalar-Poisson algebra and $( \mathcal{A},\cdot )$ is a Kokoris
algebra. Let $( \mathcal{A},\cdot )$ have an idempotent $e$ and ${\rm Ann}(\mathcal{A}, [\cdot,\cdot]) = 0$. Then  $[e ,x]=[e \circ e,x]=0$ for all $x$ in $(\mathcal{A},\cdot)$.  Therefore, $e \in {\rm Ann} (\mathcal{A}, [\cdot,\cdot])$, i.e. $e=0$. Moreover, since $(\mathcal{A}, \cdot)$ is a noncommutative Jordan algebra, $(\mathcal{A}, \cdot)$ is power-associative. Every finite-dimensional power-associative algebra which is not a nilalgebra contains a nonzero idempotent. So, it is a nilalgebra.
\end{proof}

\begin{definition}\label{knice}
For a positive integer $k$, an algebra $(\mathcal{A}, \cdot)$ is called {$k$-nice} if 
the product of any $k$ elements does not depend on their association or order.
\end{definition}
 
\begin{proposition}\label{3nice}
Let $(\mathcal{A}, \cdot)$ satisfy the identities $\eqref{sc1}$ and $\eqref{sc2}$.
If $(\mathcal{A}, \cdot)$ is $3$-nice, then it is associative$;$ the converse is not true.
\end{proposition}

\begin{proof}
Clearly, if $(\mathcal{A}, \cdot)$ is $3$-nice, then $(\mathcal{A}, \cdot)$ is associative.
By Propositions \ref{sc} and \ref{delta}, $(\mathcal{A}, \cdot)$ satisfies \eqref{id-delta-poisson} for every $\delta\neq0$; hence, by Proposition \ref{ass-shift}, an associative algebra satisfying \eqref{sc1} and \eqref{sc2} is cyclic associative. Conversely, by Corollary~12 and Proposition \ref{sc}, each cyclic associative algebra satisfies \eqref{sc1} and \eqref{sc2}.
Now let $\mathcal{A}$ be the algebra with the basis $x,y,z$, $ab$ ($a,b\in\{x,y,z\}$, $a\neq b$), $u$, $v$, where the product of two distinct elements $a,b\in\{x,y,z\}$ is $ab$, 
$(ab)c=a(bc)=u$ if $(a,b,c)$ is an even permutation of $(x,y,z)$, $(ab)c=a(bc)=v$ if $(a,b,c)$ is an odd permutation of $(x,y,z)$, and all other products of basis elements are zero.
Then $\mathcal{A}$ is cyclic associative; hence, it satisfies \eqref{sc1} and \eqref{sc2}. However, $x(yz)=u\neq v=z(yx)$, i.e., $\mathcal{A}$ is not $3$-nice.
\end{proof}
 
\begin{corollary}
    Let $(\mathcal{A}, \cdot)$ satisfy the identities \eqref{sc1} and \eqref{sc2}.
Then $(\mathcal{A}, \cdot)$ is associative if and only if it is cyclic associative.
\end{corollary}

\subsection{Transposed $\delta$-Poisson algebras in one multiplication}

\begin{proposition} \label{delta-trans}
Consider an algebra $\left( \mathcal{A},\cdot \right) $ whose underlying
vector space $\mathcal{A}$ is endowed with the two new products $\circ $ and 
$\left[ \cdot ,\cdot \right] $. If $\delta \neq 1$, then $(\mathcal{A},\circ
,\left[ \cdot ,\cdot \right] )$ is a transposed $\delta$-Poisson algebra if and
only if the algebra $\left( \mathcal{A},\cdot \right) $ satisfies the
following identities\footnote{The case  $\delta =\frac{1}{2}$ was considered by  Dzhumadildaev in \cite{dzhuma}. 
These algebras were called weak Leibniz algebras.}$:$%
\begin{equation}\label{id1-delta-trans}\adjustbox{max width=0.92\linewidth}{$F_{\delta }\left( x,y,z\right)   := 
\left( zx\right) y-\left( zy\right)x+
\left( 2\delta -1\right) \big(x\left( zy\right) - y\left( zx\right)\big)+
\delta \big(z\left( yx\right)-  z\left( xy\right)\big) \  =\ 0,$}\end{equation}
\begin{equation}\label{id2-delta-trans}\adjustbox{max width=0.92\linewidth}{$G_{\delta }\left( x,y,z\right)   : = 
\left( yz\right) x-\left( zy\right)x+
\left( 2\delta -1\right) \big(z\left( xy\right)-y\left( xz\right) \big)
 -\left( 2\delta +1\right) \big(y\left( zx\right) - z\left( yx\right) \big) \ =\ 0.$}\end{equation}
\end{proposition}

\begin{proof}
    Suppose first that $(\mathcal{A},\circ ,\left[ \cdot ,\cdot \right] )$ is a transposed ${%
\delta }$-Poisson algebra. Since $(\mathcal{A},\circ )$ is associative, we
have the identity $\left( x\circ y\right) \circ z=x\circ \left( y\circ
z\right) $ which is equivalent to:%
\begin{center}\adjustbox{max width=\linewidth}{\begin{tabular}{rclcl}
$h_{1}(x,y,z)$&$:=$&$(xy)z+(yx)z-(zy)x-(yz)x-x(yz)-x(zy)+z(yx)+z(xy)$&$=$&$0.$
\end{tabular}}\end{center}%
Also, since $(\mathcal{A},\left[ \cdot ,\cdot \right] )$ is a Lie algebra, the
Jacobi identity gives:%
\begin{center}\adjustbox{max width=\linewidth}{\begin{tabular}{rclcrcrc}
$h_{2}(x,y,z)$&$:=$&$
(xy)z-(yx)z-(zy)x-(xz)y+(yz)x+(zx)y$\\
&&$-x(yz)+y(xz)+z(yx)+x(zy)-y(zx)-z(xy)$&$=$&$0.$
\end{tabular}}\end{center}%
Finally, the identity $x\circ \left[ y,z\right] -\delta 
\big(\left[ x\circ y,z\right] + \left[ y,x\circ z\right] \big) \ =\ 0$ gives:%
\begin{center}\adjustbox{max width=\linewidth}{\begin{tabular}{lcl}
$h_{4}^{\delta }(x,y,z)$&$:=$&$
x\left( yz\right) -x\left( zy\right) +\left(
yz\right) x-\left( zy\right) x-\delta \left( xy\right) z-\delta \left(
yx\right) z+\delta z\left( xy\right) +$\\
\multicolumn{3}{r}{$\delta z\left( yx\right) -\delta
y\left( xz\right) -\delta y\left( zx\right) +\delta \left( xz\right)
y+\delta \left( zx\right) y.$}
\end{tabular}}\end{center}
Now it is straightforward to verify that
\begin{center}\adjustbox{max width=\linewidth}{\begin{tabular}{lcl}
$2F_{\delta }\left( x,y,z\right)$  &$=$&$\left( 1-\delta \right)
h_{1}(x,z,y)+\delta h_{2}(x,y,z)+\left( 1-\delta \right) h_{4}^{\delta
}(x,y,z)+$\\
\multicolumn{3}{r}{$\left( \delta -1\right) h_{4}^{\delta }(y,x,z)-\delta h_{4}^{\delta
}(z,x,y),$} \\
$G_{\delta }\left( x,y,z\right)  $&$=$&$\delta h_{1}(x,y,z)-\delta h_{1}(x,z,y)+%
\frac{1}{2}h_{2}(x,y,z)+$\\
\multicolumn{3}{r}{$\frac{1}{2}h_{4}^{\delta }(x,y,z)+\frac{1}{2}%
h_{4}^{\delta }(y,x,z)-\frac{1}{2}h_{4}^{\delta }(z,x,y).$}
\end{tabular}}\end{center}
On the other hand, if  $(\mathcal{A}, \cdot)$ satisfies the identities (\ref{id1-delta-trans}) and  (\ref{id2-delta-trans}) then we have
\begin{center}\adjustbox{max width=\linewidth}{\begin{tabular}{lcl}
$h_{1}\left( x,y,z\right)  $&$=$&$-\frac{2}{3}F_{\delta }\left( x,y,z\right) -%
\frac{4}{3}F_{\delta }\left( x,z,y\right) -\frac{2}{3}F_{\delta }\left(
y,z,x\right) -$\\
\multicolumn{3}{r}{$\frac{1}{3}G_{\delta }\left( x,y,z\right) +\frac{2}{3}%
G_{\delta }\left( y,z,x\right) +\frac{1}{3}G_{\delta }\left( z,y,x\right) ,$}
\\
$h_{2}\left( x,y,z\right)  $&$=$&$\frac{3\delta -2}{3\left( \delta -1\right) }%
G_{\delta }\left( x,y,z\right) +\frac{3\delta -2}{3\left( \delta -1\right) }%
G_{\delta }\left( y,z,x\right) \allowbreak -\frac{3\delta -2}{3\left( \delta
-1\right) }G_{\delta }\left( z,y,x\right) -$\\
\multicolumn{3}{r}{$\frac{1}{3\left( \delta -1\right) 
}F_{\delta }\left( x,y,z\right) +\frac{1}{3\left( \delta -1\right) }%
F_{\delta }\left( x,z,y\right) -\frac{1}{3\left( \delta -1\right) }F_{\delta
}\left( y,z,x\right) ,$} \\
$h_{4}^{\delta }\left( x,y,z\right)  $&$=$&$-\frac{2\delta ^{2}-5\delta +2}{3\left( \delta -1\right)}%
G_{\delta }\left( x,y,z\right) +\frac{\delta ^{2}-\delta +1}{3\left( \delta -1\right)}%
G_{\delta }\left( y,z,x\right) -\frac{\delta ^{2}-\delta +1}{3\left( \delta -1\right)}%
G_{\delta }\left( z,y,x\right) +$\\
\multicolumn{3}{r}{$\frac{2\delta ^{2}-2\delta -1}{3\left( \delta -1\right)}%
F_{\delta }\left( x,y,z\right) -\allowbreak \frac{2\delta ^{2}-2\delta -1}{%
3\left( \delta -1\right)}F_{\delta }\left( x,z,y\right) -\frac{\left( 2\delta -1\right)
^{2}}{3\left( \delta -1\right)}F_{\delta }\left( y,z,x\right) .$}
\end{tabular}}\end{center}
Now we consider the following algebras:
\begin{center}\adjustbox{max width=\linewidth}{\begin{tabular}{lllllll}
$\mathcal{B}_1$ &$:$&$e_{1}\cdot e_{3}=e_{1}.$ \\
$\mathcal{B}_2$ &$:$&$e_{1}\cdot e_{1}=e_{1},$&$e_{1}\cdot e_{3}=e_{2}.$
\end{tabular}}\end{center}%
\noindent  $\mathcal{B}_1$ satisfies ($\ref{id1-delta-trans}$) but it does not satisfy ($\ref{id2-delta-trans}$).  $\mathcal{B}_2$
satisfies ($\ref{id2-delta-trans}$) but it does not satisfy ($\ref{id1-delta-trans}$).
Hence, identities ($\ref{id1-delta-trans}$) and  ($\ref{id2-delta-trans}$) are independent.
\end{proof}

\begin{proposition}
Consider an algebra $\left( \mathcal{A},\cdot \right) $ whose underlying
vector space $\mathcal{A}$ is endowed with the two new products $\circ $ and 
$\left[ \cdot ,\cdot \right] $. Then $(\mathcal{A},\circ ,%
\left[ \cdot ,\cdot \right] )$ is a transposed $1$-Poisson algebra if
and only if the algebra $\left( \mathcal{A},\cdot \right) $ satisfies the
following identities$:$%
\begin{eqnarray}
F_{1}\left( x,y,z\right)  &:&= \ \left( zx\right) y-\left( zy\right) x+x\left(
zy\right) -z\left( xy\right) -y\left( zx\right) +z\left( yx\right) =0, \label{id1-td1}\\
G_{1}\left( x,y,z\right)  &:&= \ \left( yz\right) x-\left( zy\right) x-y\left(
xz\right) +z\left( xy\right) -3y\left( zx\right) +3z\left( yx\right) =0, \label{id2-td1}\\
H\left( x,y,z\right)  &:&= \ x(yz)-x(zy)+y(zx)-y(xz)+z(xy)-z(yx)=0.\label{id3-td1}
\end{eqnarray}
\end{proposition}
\begin{proof}
Suppose first that $(\mathcal{A},\circ ,\left[ \cdot ,\cdot \right] )$ is a transposed $1$-Poisson algebra. Then, as in the proof of Proposition \ref{delta-trans}, we have 
\begin{center}\adjustbox{max width=\linewidth}{\begin{tabular}{lclcl}
$F_{1}\left( x,y,z\right)  $&$=$&$\frac{1}{2}h_{2}(x,y,z)-\frac{1}{2}h_{4}^{1}(z,x,y),$ \\
$G_{1}\left( x,y,z\right)  $&$=$&
$h_{1}(x,y,z)-h_{1}(x,z,y)+\frac{1}{2}h_{2}(x,y,z)+$\\
\multicolumn{3}{r}{$\frac{1}{2}h_{4}^{1}(x,y,z)+\frac{1}{2}h_{4}^{1 }(y,x,z)-%
\frac{1}{2}h_{4}^{1}(z,x,y),$} \\
$H\left( x,y,z\right)  $&$=$&$-\frac{1}{2}h_{2}(x,y,z)+\frac{1}{2}%
h_{4}^{1}(x,y,z)-\frac{1}{2}h_{4}^{1}(y,x,z)+\frac{1}{2}h_{4}^{1}(z,x,y).$
\end{tabular}}\end{center}%
On the other hand, if  $(\mathcal{A},\cdot)$ satisfies the identities ($\ref{id1-td1}$), ($\ref{id2-td1}$) and  ($\ref{id3-td1}$) then we have
\begin{center}\adjustbox{max width=\linewidth}{\begin{tabular}{lclcl}
$h_{1}\left( x,y,z\right)  $&$=$&$-\frac{2}{3}F_{1}\left( x,y,z\right) -%
\frac{4}{3}F_{1 }\left( x,z,y\right) -\frac{2}{3}F_{1 }\left(
y,z,x\right) -$\\
\multicolumn{3}{r}{$\frac{1}{3}G_{1 }\left( x,y,z\right) +\frac{2}{3}%
G_{1 }\left( y,z,x\right) +\frac{1}{3}G_{1 }\left( z,y,x\right)$}&$=$&$0,$\\
$h_{2}\left( x,y,z\right)  $&$=$&$G_{1}\left( x,y,z\right) +G_{1}\left(
y,z,x\right) -G_{1}\left( z,y,x\right) +H\left( x,y,z\right)$&$=$&$0,$\\
$h_{4}^{1}\left( x,y,z\right)  $&$=$&$F_{1 }\left( x,y,z\right) -F_{1
}\left( x,z,y\right) -F_{1 }\left( y,z,x\right) +H\left( x,y,z\right)$&$=$&$0.$
\end{tabular}}\end{center}
Now we consider the following algebras:

\begin{center}\adjustbox{max width=\linewidth}{\begin{tabular}{lllllllll}
$\mathcal{B}_1$ &$:$&$e_{1}\cdot e_{3}=e_{1}.$ \\
$\mathcal{B}_2$ &$:$&$e_{1}\cdot e_{1}=e_{1},$&$e_{1}\cdot e_{3}=e_{2}.$\\ 
$\mathcal{B}_{3}$ &$:$&
$e_{3}\cdot e_{3}=e_{4},$&$e_{1} \cdot e_{2}=e_{4}+e_{5},$&$e_{2}\cdot e_{1}=e_{5}-e_{4},$&$e_{1}\cdot e_{3}=-e_{1},$\\
&&$e_{3}\cdot e_{1}=e_{1},$&$e_{2}\cdot e_{3}=e_{2},$&$e_{3}\cdot e_{2}=-e_{2},$&$e_{4}\cdot e_{3}=2e_{5}.$\\
\end{tabular}}\end{center}%
\noindent  $\mathcal{B}_1$ satisfies ($\ref{id1-td1}$) and ($\ref{id3-td1}$) but it does not satisfy ($\ref{id2-td1}$). 
    $\mathcal{B}_2$ satisfies ($\ref{id2-td1}$) and ($\ref{id3-td1}$) but it does not satisfy ($\ref{id1-td1}$).
$\mathcal{B}_{3}$ satisfies   ($\ref{id1-td1}$) and ($\ref{id2-td1}$) but it does not satisfy ($\ref{id3-td1}$). Hence, the identities ($\ref{id1-td1}$), ($\ref{id2-td1}$) and  ($\ref{id3-td1}$) are independent. \end{proof}

\subsection{Transposed scalar-Poisson algebras in one multiplication}

\begin{definition}
An algebra $(\mathrm{P},\cdot ,\{\cdot ,\cdot \})$ is defined to be {%
transposed scalar-Poisson algebra}, if $(\mathrm{P},\cdot )$ is a commutative
associative algebra, $(\mathrm{P},\{\cdot ,\cdot \})$ is a Lie algebra and the following identities hold$:$ 
\begin{center}
$x\{y,z\} \  = \ 0, \  \{xy,z\}+\{y,xz\}\ =\ 0.$
\end{center}
\end{definition}
It is clear that any mixed-Poisson algebra $(\mathrm{P},\cdot ,\{\cdot ,\cdot \})$ is a scalar-Poisson and transposed scalar-Poisson algebra. Also, if $(\mathrm{P},\cdot ,\{\cdot ,\cdot \})$ is a scalar-Poisson and transposed scalar-Poisson algebra, then it is a mixed-Poisson algebra. That is to say, the variety of mixed-Poisson algebras is the intersection of the varieties of scalar-Poisson algebras and transposed scalar-Poisson algebras.

 \begin{proposition}\label{tsc}
  Consider an algebra $\left( \mathcal{A},\cdot \right) $ whose underlying
vector space $\mathcal{A}$ is endowed with the two new products $\circ $ and 
$\left[ \cdot ,\cdot \right] $. Then $(\mathcal{A},\circ ,\left[ \cdot
,\cdot \right] )$ is a {transposed scalar-Poisson algebra if and only if the }algebra $%
\left( \mathcal{A},\cdot \right) $ satisfies the following identities$:$%
\begin{eqnarray}
S_{1}\left( x,y,z\right) \ &:=& \ 2(xy)z-2(zy)x+z(xy)-2y(zx)+z(yx)\ =\ 0,
\label{tsc1} \\
S_{2}\left( x,y,z\right) \ &:=& \ (yz)x-(zy)x-2y(zx)+2z(yx)\ =\ 0.  \label{tsc2}
\end{eqnarray}
\end{proposition}

\begin{proof}
    Let $h_{1}(x,y,z)$, $h_{2}(x,y,z)$ and $h_{4}^{0}(x,y,z)$ be as defined in Proposition \ref{delta-trans}. The identity $\left[ x\circ y,z\right] +\left[ y,x\circ z\right]=0 $ is equivalent to: 
\begin{center}\adjustbox{max width=\linewidth}{\begin{tabular}{lclcl}
$D(x,y,z)$&$:=$&$\left( xy\right) z-\left( xz\right) y+\left( yx\right) z-\left(
zx\right) y+y\left( xz\right) -z\left( xy\right) +y\left( zx\right) -z\left(
yx\right) $&$=$&$0.$
\end{tabular}}\end{center}
Suppose first that $(\mathcal{A},\circ ,\left[ \cdot
,\cdot \right] )$ is a transposed scalar-Poisson algebra. Then it is straightforward to verify that:
\begin{center}\adjustbox{max width=\linewidth}{\begin{tabular}{lcl}
$S_{1}\left( x,y,z\right)  $&$=$&$
h_{1}(x,y,z)-
\frac{1}{2}h_{1}(x,z,y)+
\frac{1}{2}h_{2}(x,y,z)+
\frac{1}{2}h_{4}^{0}(x,y,z)+$\\
\multicolumn{3}{r}{$
\frac{1}{2}h_{4}^{0}(y,x,z)+
\frac{1}{2}h_{4}^{0}(z,x,y)-
\frac{1}{2}D(x,y,z)+
\frac{1}{2}D(x,y,z)=0, $}\\
$S_{2}\left( x,y,z\right)  $&$=$&$
\frac{1}{2}h_{1}(x,y,z)-
\frac{1}{2}h_{1}(x,z,y)+
\frac{1}{2}h_{2}(x,y,z)+
\frac{1}{2}h_{4}^{0}(x,y,z)+$\\
\multicolumn{3}{r}{$
\frac{1}{2}h_{4}^{0}(y,x,z)-
\frac{1}{2}h_{4}^{0}(z,x,y)-
\frac{1}{2}D(x,y,z)=0.$} \\
\end{tabular}}\end{center}%
On the other hand, if  $(\mathcal{A},\cdot)$  satisfies the identities ($\ref{tsc1}$) and  \eqref{tsc2} then we have
\begin{center}\adjustbox{max width=\linewidth}{\begin{tabular}{lclcl}
$h_{1}\left( x,y,z\right)  $&$=$&$
\frac{2}{3}S_{1}\left( x,y,z\right)+S_{1}\left( y,x,z\right) +\frac{2}{3}S_{1}\left( z,x,y\right) -\frac{1}{3}%
S_{1}\left( y,z,x\right) -$\\
\multicolumn{3}{r}{$\frac{1}{3}S_{2}\left( x,y,z\right) -\frac{2}{3}%
S_{2}\left( y,x,z\right) -\frac{1}{3}S_{2}\left( z,x,y\right)$}&$=$&$0,$ \\
$h_{2}\left( x,y,z\right)  $&$=$&$
\frac{1}{3}S_{1}\left( x,y,z\right) +\frac{1}{3}%
S_{1}\left( z,x,y\right) +\frac{1}{3}S_{1}\left( y,z,x\right) +$\\
\multicolumn{3}{r}{$
\frac{1}{3}%
S_{2}\left( x,y,z\right) -\frac{1}{3}S_{2}\left( y,x,z\right) +\frac{1}{3}%
S_{2}\left( z,x,y\right)$}&$=$&$0,$ \\

$h_{4}^{0}(x,y,z)  $&$=$&$
\frac{1}{3}S_{1}\left( x,y,z\right) +\frac{2}{3}S_{1}\left(
x,z,y\right) +\frac{1}{3}S_{1}\left( y,x,z\right) +\frac{2}{3}S_{1}\left(
z,x,y\right) +$\\
\multicolumn{3}{r}{$
S_{1}\left( y,z,x\right)  +\frac{1}{3%
}S_{2}\left( x,y,z\right) +\frac{2}{3}S_{2}\left( y,x,z\right) -\frac{2}{3}%
S_{2}\left( z,x,y\right)$}&$=$&$0,$ \\
$D\left( x,y,z\right)  $&$=$&$
\frac{1}{3}S_{1}\left( x,y,z\right) +
\frac{1}{3}S_{1}\left( x,z,y\right) +
\frac{2}{3}S_{1}\left( y,x,z\right) +
\frac{2}{3}S_{1}\left( y,z,x\right) -$\\
\multicolumn{3}{r}{$
\frac{2}{3}S_{2}\left( x,y,z\right) -
\frac{1}{3}S_{2}\left( y,x,z\right) +
\frac{1}{3}S_{2}\left( z,x,y\right)$}&$=$&$0.$ \\
\end{tabular}}\end{center}

Now we consider the following algebras:

\begin{center}\adjustbox{max width=\linewidth}{\begin{tabular}{lllllllll}

$\mathcal{B}_1$ &$:$& $e_{1}\cdot e_{1}=e_{1},$&$e_{1}\cdot e_{3}=e_{2}.$\\ 
$\mathcal{B}_{2}$ &$:$& $e_{1}\cdot e_{1}=e_{1},$&$e_{1}\cdot e_{2}=-2e_{2},$&$e_{2}\cdot e_{1}=e_{2}.$\\
\end{tabular}}\end{center}%
\noindent $\mathcal{B}_1$ satisfies (\ref{tsc2}) but it does not satisfy (\ref{tsc1}).  $\mathcal{B}_{2}$ satisfies   (\ref{tsc1})  but it does not satisfy (\ref{tsc2}). Hence, the identities (\ref{tsc1}) and  (\ref{tsc2}) are independent. 
\end{proof}

\subsection{Mixed-Poisson algebras in one multiplication}

Let us consider the intersection of the varieties of
$\delta_1$-Poisson and transposed $\delta_2$-Poisson algebras.
Due to Proposition \ref{dptm} and Example \ref{delta-mix-ex}, there are the following two situations.

\begin{definition}
An algebra $({\rm P}, \cdot, \{\cdot,\cdot\})$ is called a {mixed-Poisson algebra} 
if $({\rm P}, \cdot)$ is a  commutative associative algebra,
$({\rm P},  \{\cdot,\cdot\})$ is a Lie algebra and the following
identities hold:%
\[\adjustbox{max width=\linewidth}{$\displaystyle \{x  y, z \} \ =\  x  \{ y,z \} \ = \ 0$}\]
or, equivalently, the identity $\{xy,z\}+\{x,z\}y-\{y,z\}x\ =\ 0$ holds.
 
\end{definition}
 
\begin{definition}
An algebra $({\rm P}, \cdot, \{\cdot,\cdot\})$ is called a $\delta$-{mixed-Poisson algebra} 
if $({\rm P}, \cdot)$ is a  commutative associative algebra,
$({\rm P},  \{\cdot,\cdot\})$ is a Lie algebra and the following
identities hold:%
\begin{center}\adjustbox{max width=\linewidth}{\begin{tabular}{lcl} 
$\{x  y, z \} $&$ =$&$  \delta \big( x  \{ y,z \} + \{ x,z \} y \big),$\\
$x \{ y, z \} $&$ = $&$  \frac{1}{3\delta} \big( \{x  y,z \} + \{ y,xz \} \big).$
\end{tabular}}\end{center}

\end{definition}

\begin{proposition}\label{depol}
    Consider an algebra $\left( \mathcal{A},\cdot \right) $ whose underlying
vector space $\mathcal{A}$ is endowed with the two new products $\circ $ and 
$\left[ \cdot ,\cdot \right] $. 
Then $(\mathcal{A},\circ ,
\left[ \cdot ,\cdot \right] )$ satisfies $\left[ x\circ y,z\right] =x\circ \left[ y,z\right] =0$ if and only if $\left( \mathcal{A},\cdot \right) $ satisfies the following identity:
\begin{equation}
L(x,y,z)\ :=\ (xz)y-y(zx)\ =\ 0.  \label{id1-mix}
\end{equation}
Furthermore,  $(\mathcal{A},\circ ,
\left[ \cdot ,\cdot \right] )$ is a mixed-Poisson algebra if and only if the algebra $\left( \mathcal{A},\cdot \right) $ satisfies the identities \eqref{tsc2} and \eqref{id1-mix}.
\end{proposition}

\begin{proof}
Let $h_{1}(x,y,z)$, $h_{2}(x,y,z)$, $h_{3}^{0 }(x,y,z)$ and $h_{4}^{0}(x,y,z)$ be as defined in Propositions \ref{delta} and  \ref{delta-trans}. Then we have

\begin{center}\adjustbox{max width=\linewidth}{\begin{tabular}{lclcl}
$L\left( x,y,z\right)  $&$=$&$\frac{1}{2}h_{3}^{0}(x,z,y)+\frac{1}{2}%
h_{4}^{0}(y,x,z),$ \\
$S_{2}\left( x,y,z\right)  $&$=$&$
\frac{1}{2}h_{1}(x,y,z)-\frac{1}{2}h_{1}(x,z,y)+%
\frac{1}{2}h_{2}(x,y,z)-\frac{1}{2}h_{3}^{0}(x,y,z)+\frac{1}{2}%
h_{3}^{0}(x,z,y)+$\\
\multicolumn{3}{r}{$\frac{1}{2}h_{4}^{0}(x,y,z)+\frac{1}{2}h_{4}^{0}(y,x,z)-%
\frac{1}{2}h_{4}^{0}(z,x,y),$}
\end{tabular}}\end{center}
On the other hand, we have 
\begin{center}\adjustbox{max width=\linewidth}{\begin{tabular}{lclcl}
$h_{1}(x,y,z) $&$=$&$-\frac{4}{3}L\left( x,y,z\right) +\frac{1}{3}%
L\left( x,z,y\right) -\frac{5}{3}L\left( y,x,z\right) -\frac{1}{3}L\left(
z,x,y\right) +\frac{5}{3}L\left( y,z,x\right) +$\\
\multicolumn{3}{r}{$\frac{4}{3}L\left(
z,y,x\right) +\frac{2}{3}S_{2}\left( x,y,z\right) +\frac{4}{3}S_{2}\left(
y,x,z\right) +\frac{2}{3}S_{2}\left( z,x,y\right) ,$} \\
$h_{2}\left( x,y,z\right)  $&$=$&$-\frac{1}{3}L\left( x,y,z\right) +\frac{1}{3}%
L\left( x,z,y\right) +\frac{1}{3}L\left( y,x,z\right) -\frac{1}{3}L\left(
z,x,y\right) -\frac{1}{3}L\left( y,z,x\right) + $\\
\multicolumn{3}{r}{$\frac{1}{3}L\left(
z,y,x\right) +\frac{2}{3}S_{2}\left( x,y,z\right) -\frac{2}{3}S_{2}\left(
y,x,z\right) +\frac{2}{3}S_{2}\left( z,x,y\right) ,$} \\
$h_{3}^{0}(x,y,z) $&$=$&$L\left( x,z,y\right) +L\left( y,z,x\right) ,$ \\
$h_{4}^{0}(x,y,z) $&$=$&$L\left( y,x,z\right) -L\left( z,x,y\right) .$ \\
\end{tabular}}\end{center}

Therefore, $(\mathcal{A},\circ ,
\left[ \cdot ,\cdot \right] )$ satisfies $\left[ x\circ y,z\right] =x\circ \left[ y,z\right] =0$ (i.e. $h_{3}^{0}(x,y,z)=h_{4}^{0}(x,y,z)=0$) if and only if $\left( \mathcal{A},\cdot \right) $ satisfies \eqref{id1-mix} and $(\mathcal{A},\circ ,
\left[ \cdot ,\cdot \right] )$ is a mixed-Poisson algebra if and only if the algebra $\left( \mathcal{A},\cdot \right) $ satisfies the identities \eqref{tsc2} and \eqref{id1-mix}. 

Now we consider the following algebras:
\begin{center}\adjustbox{max width=\linewidth}{\begin{tabular}{lllllll}
$\mathcal{B}_1$ &$:$&$e_{1}\cdot e_{1}=e_{2},$&$e_{1}\cdot e_{2}=e_{2}.$ \\
$\mathcal{B}_2$ &$:$&$e_{1}\cdot e_{1}=e_{2},$&$e_{2}\cdot e_{2}=e_{1}.$
\end{tabular}}\end{center}%
\noindent  $\mathcal{B}_1$ satisfies \eqref{tsc2} but it does not satisfy \eqref{id1-mix}.  $\mathcal{B}_2$
satisfies \eqref{id1-mix} but it does not satisfy \eqref{tsc2}.
Hence, identities \eqref{tsc2} and  \eqref{id1-mix} are independent.
\end{proof}

\begin{proposition}
    Consider an algebra $\left( \mathcal{A},\cdot \right) $ whose underlying
vector space $\mathcal{A}$ is endowed with the two new products $\circ $ and 
$\left[ \cdot ,\cdot \right] $. 
Then $(\mathcal{A},\circ ,
\left[ \cdot ,\cdot \right] )$ is a $\delta$-mixed-Poisson algebra if and only
if the algebra $\left( \mathcal{A},\cdot \right) $ satisfies the identities $\eqref{id-delta-poisson}$ and \eqref{id3-td1}.
\end{proposition}
\begin{proof}
   Let $h_{2}(x,y,z)$, $h_{3}^{\delta }(x,y,z)$ and $h_{4}^{\frac{1}{3\delta}}(x,y,z)$ be as defined in Proposition \ref{delta} and Proposition \ref{delta-trans}. Due to Proposition \ref{delta}, we just need to show that if $\left( \mathcal{A},\cdot \right) $ satisfies the identities ($\ref{id-delta-poisson}$) and (\ref{id3-td1}) then $x\circ \left[ y,z\right] =  \frac{1}{3\delta} \big( \left[ x\circ y,z%
\right]+ \left[ y,x\circ z\right] \big)$, i.e. $h_{4}^{\frac{1}{3\delta}}(x,y,z)=0$, and if $(\mathcal{A},\circ ,
\left[ \cdot ,\cdot \right] )$ is a $\delta$-mixed-Poisson algebra then $\left( \mathcal{A},\cdot \right) $ satisfies the identity (\ref{id3-td1}). Indeed, it is straightforward to verify that:
\begin{center}\adjustbox{max width=\linewidth}{\begin{tabular}{lcl}
$H\left( x,y,z\right)  $&$=$&$-\frac{1}{2}h_{2}(x,y,z)+\frac{1}{2\delta }%
h_{3}^{\delta }(x,y,z)-\frac{1}{2\delta }h_{3}^{\delta }(x,z,y)+\frac{3}{2}%
h_{4}^{\frac{1}{3\delta }}(x,y,z),$\\
$h_{4}^{\frac{1}{3\delta }}(x,y,z) $&$=$&$
-\frac{1}{9\delta ^{2}} \big( 
f_{\delta}(x,y,z)-f_{\delta }(x,z,y)+f_{\delta }(y,x,z)-f_{\delta }(z,x,y)\big)+$\\
\multicolumn{3}{r}{$\frac{1}{3\delta 
}f_{\delta }(y,z,x)-\frac{1}{3\delta }f_{\delta }(z,y,x)+\frac{2}{3}H\left(
x,y,z\right).$}
\end{tabular}}\end{center}

Now we consider the following algebras:
\begin{center}\adjustbox{max width=\linewidth}{\begin{tabular}{lllllll}
$\mathcal{B}_1$ &$:$&$e_{3}\cdot e_{3}=e_{4},$&$e_{1} \cdot e_{2}=e_{4}+e_{5},$&$e_{2}\cdot e_{1}=e_{5}-e_{4},$&$e_{1}\cdot e_{3}=-e_{1},$\\
&&$e_{3}\cdot e_{1}=e_{1},$&$e_{2}\cdot e_{3}=e_{2},$&$e_{3}\cdot e_{2}=-e_{2},$&$e_{3}\cdot e_{4}=e_{4}\cdot e_{3}=e_{5}.$\\

$\mathcal{B}_2$ &$:$&$e_{1}\cdot e_{1}=e_{1},$&$ e_{1}\cdot e_{2}=e_{2}.$
\end{tabular}}\end{center}%
\noindent  $\mathcal{B}_1$ satisfies ($\ref{id-delta-poisson}$) but it does not satisfy ($\ref{id3-td1}$).  
$\mathcal{B}_2$
satisfies ($\ref{id3-td1}$) but it does not satisfy ($\ref{id-delta-poisson}$).
Hence, identities ($\ref{id-delta-poisson}$) and ($\ref{id3-td1}$) are independent.
\end{proof}
\section{Free   algebras}\label{free}

 The description of bases of free algebras in varieties of nonassociative algebras is one of the classical problems.
 The first results in this direction were obtained by Hall and Shirshov, who found bases of free Lie algebras.
 Later, Chibrikov described a right-normed basis of free Lie algebras \cite{Chibrikov}.
 There are many results on bases of free algebras in other nonassociative varieties (see, for example, \cite{B}).
 A basis of free Poisson algebras was constructed in a paper by Shestakov \cite{sh93}. As we will see in the present section, 
 free $\delta$-Poisson algebras admit a basis similar to the one constructed in \cite{sh93}.

\subsection{Free $\delta$-Poisson algebras} 

In this section, we construct a basis of the free $\delta$-Poisson algebra $\delta$-${\rm P}(X)$ generated by a countable set  $X$ for  $\delta\neq 0,1$.  

\begin{lemma}\label{0id}
The algebra $\delta$-${\rm P}(X)$ satisfies the following identities$:$
$$\{a,bcd\}=ab\{c,d\}=a\{b,cd\}=\{ab,\{c,d\}\}=0$$
and
$$\{a,b\{c,\{d,e\}\}\}=\{a,\{b,c\{d,e\}\}\}=\{a,\{b,\{c,de\}\}\}=a\{b,\{c,de\}\}=0.$$
\end{lemma}
\begin{proof}
The first part was proved in Proposition \ref{antixyzt}. The remaining identities can be verified by computer algebra, for example, using the software \cite{DotsHij}.
\end{proof}

We set
\begin{center}\adjustbox{max width=\linewidth}{\begin{tabular}{lcl}
$\mathcal{B}_1$&$=$&$\langle x_i\rangle,$\\
$\mathcal{B}_2$&$=$&$\langle x_{i_1}x_{i_2},\{x_{j_1},x_{j_2}\}\;|\;i_1\leq i_2;j_1<j_2\rangle,$\\
$\mathcal{B}_3$&$=$&$\langle x_{i_1}x_{i_2}x_{i_3},\;\{x_{j_1},\{x_{j_2},x_{j_3}\}\},\;x_{k_1}\{x_{k_2},x_{k_3}\}\;|\;
i_1\leq i_2\leq i_3;\;j_1\leq j_3,\;j_2<j_3;\;k_2<k_3\rangle,$\\
$\mathcal{B}_4$&$=$&$\langle x_{i_1}x_{i_2}x_{i_3}x_{i_4},\;x_{j_1}\{x_{j_2},\{x_{j_3},x_{j_4}\}\},\; \{x_{k_1},x_{k_2}\}\{x_{k_3},x_{k_4}\},\; \mathcal{C}_4 \;|\;$\\
\multicolumn{3}{r}{$
i_1\leq i_2\leq i_3\leq i_4;\;j_1\leq j_2,\;j_1\leq j_3,\;j_1<j_4,\;j_2\leq j_4,\;j_3<j_4;\;$}\\
\multicolumn{3}{r}{$
k_1< k_2,\; k_3<k_4,\; k_1 < k_3 \mbox{ or } k_1=k_3 \mbox{ and } k_2\leq k_4\rangle,$}
\end{tabular}}\end{center}
\noindent where $\mathcal{C}_{n}$ is a right-normed basis of free Lie algebra of degree $n$ which is given in  \cite{Chibrikov}.

For $n\geq 5$, let us define three types of monomials 
\begin{center}\adjustbox{max width=\linewidth}{\begin{tabular}{lcl}
$\mathcal{B}_{n}^1$&$=$&$\langle\{x_{i_1},\{x_{i_2},\{\cdots \{x_{i_{n-1}},x_{i_n}\}\cdots\}\}\}\rangle=\mathcal{C}_n,$\\
$\mathcal{B}_{n}^2$&$=$&$\langle x_{i_1}\{x_{i_2},\{\cdots \{x_{i_{n-1}},x_{i_n}\}\cdots\}\}\;|\; i_1\leq i_2,\ldots, i_{n-1},i_n\;\text{and}\;\{x_{i_2},\{\cdots \{x_{i_{n-1}},x_{i_n}\}\}\}\in\mathcal{C}_{n-1}\rangle$\\
$\mathcal{B}_{n}^3$&$=$&$\langle x_{i_1}x_{i_2}\cdots x_{i_{n-1}}x_{i_n}\;|\; i_1\leq i_2\leq \ldots \leq i_{n-1}\leq i_n\rangle.$
\end{tabular}}\end{center}
Also, we set
$$\mathcal{B}_n=\mathcal{B}_n^1\cup\mathcal{B}_n^2\cup \mathcal{B}_n^3.$$
\begin{theorem}\label{basisAnti}
The set $\cup_i\mathcal{B}_i$ is a basis of algebra $\delta$-${\rm P}(X)$.
\end{theorem}
\begin{proof}
Firstly, let us show that any monomial of $\delta$-${\rm P}(X)$ can be written as a linear combination of the set $\bigcup_i\mathcal{B}_i$. For monomials of degree at most $4$, the result can be verified by the software \cite{DotsHij}. From degree $5$, let us show that any monomial of $\delta$-${\rm P}(X)$ can be written as a linear combination of monomials of the form 
\begin{equation}\label{rightnorm}
x_{i_1}\star(x_{i_2}\star(\cdots(x_{i_{n-1}}\star x_{i_n})\cdots)),
\end{equation}
where $\star$ is $\cdot$ or $\{\cdot,\cdot\}$. It can be proved by induction on the length of monomials. The base of induction is monomials of degree $3$. By (\ref{deltpois}) and \cite{Chibrikov}, the algebra $\delta$-${\rm P}(X)$ is spanned by monomials of the following form:
$$\{A_1\}\{A_2\}\cdots\{A_k\},$$
where $\{A_i\}$ are pure right-normed Lie monomials. If $\{A_1\}$ is a generator, then we write
$$x_{i_1}(\{A_2\}\cdots\{A_k\}),$$
and the monomial in the bracket can be rewritten in the form (\ref{rightnorm}) by inductive hypothesis.
If ${\rm deg} \ (\{A_1\})>1$, then for $\{A_1\}\{A_2\}$, we have
\begin{multline*}
\{A_1\}\{A_2\}=\{x_{i_1},\{x_{i_2},\{\cdots \{x_{i_{n-1}},x_{i_n}\}\cdots\}\}\}\{x_{j_1},\{x_{j_2},\{\cdots \{x_{j_{m-1}},x_{j_m}\}\cdots\}\}\}=\\
{\delta}^{-1}\{x_{i_1},\{x_{i_2},\{\cdots \{x_{i_{n-1}},x_{i_n}\}\cdots\}\}\{x_{j_1},\{x_{j_2},\{\cdots \{x_{j_{m-1}},x_{j_m}\}\cdots\}\}\}\}\\
-\{x_{i_1},\{x_{j_1},\{x_{j_2},\{\cdots \{x_{j_{m-1}},x_{j_m}\}\cdots\}\}\}\}\}\{x_{i_2},\{\cdots \{x_{i_{n-1}},x_{i_n}\}\cdots\}.
\end{multline*}
Without loss of generality, we may assume that $m>n$. By the inductive hypothesis, the first monomial of the right side can be written in the form (\ref{rightnorm}). The second monomial of the right side can be written in the form (\ref{rightnorm}) by repeating the given process several times. We repeat the given process with the obtained result. Finally, we obtain that any monomial can be written as a linear combination of monomials of the form (\ref{rightnorm}).

By Lemma \ref{0id}, if we have a monomial
\[\adjustbox{max width=\linewidth}{$\displaystyle \{x_{i_1},x_{i_2}\star(\cdots(x_{i_{n-1}}\star x_{i_n})\cdots)\}$}\]
and at least one operation $\star$ of this monomial is $\cdot$, then such a monomial is equal to $0$. By Lemma \ref{0id}, if we have a monomial
\[\adjustbox{max width=\linewidth}{$\displaystyle x_{i_1}\cdot(x_{i_2}\star(\cdots(x_{i_{n-1}}\star x_{i_n})\cdots))$}\]
and the operations $\star$ of this monomial include both $\cdot$ and $\{\cdot,\cdot\}$, then such a monomial is equal to $0$. The remaining nonzero monomials have the following forms:
$$\{x_{i_1},\{x_{i_2},\{\cdots \{x_{i_{n-1}},x_{i_n}\}\cdots\}\}\},\;\;x_{i_1}\{x_{i_2},\{\cdots \{x_{i_{n-1}},x_{i_n}\}\cdots\}\},\;\;x_{i_1}x_{i_2}\cdots x_{i_n}.$$
By \cite{Chibrikov}, the first set of monomials corresponds to $\mathcal{B}_n^1$.

Let us show that any monomial of the form
$$x_{i_1}\{x_{i_2},\{\cdots \{x_{i_{n-1}},x_{i_n}\}\cdots\}\}$$
can be written as a linear combination of monomials from $\mathcal{B}_n^2$. We use induction on the length of the monomial; the base of induction is $n=5$. We consider $3$ cases:
$$x_{i_2}\{x_{i_1},\{x_{i_3},\{x_{i_4},x_{i_5}\}\}\},
\;\;x_{i_2}\{x_{i_3},\{x_{i_1},\{x_{i_4},x_{i_5}\}\}\},
\;\;x_{i_2}\{x_{i_3},\{x_{i_4},\{x_{i_1},x_{i_5}\}\}\},$$
where $i_1< {\rm min} \{i_2,i_3,i_4,i_5\}$.
In the first case, by (\ref{deltpois}) and Lemma \ref{0id},
\begin{center}\adjustbox{max width=\linewidth}{\begin{tabular}{lclclc}
$x_{i_2}\{x_{i_1},\{x_{i_3},\{x_{i_4},x_{i_5}\}\}\}$&$=
$&${\delta}^{-1}\{x_{i_1}x_{i_2},\{x_{i_3},\{x_{i_4},x_{i_5}\}\}\}-x_{i_1}\{x_{i_2},\{x_{i_3},\{x_{i_4},x_{i_5}\}\}\}$&$=$\\
\multicolumn{3}{r}{$-x_{i_1}\{x_{i_2},\{x_{i_3},\{x_{i_4},x_{i_5}\}\}\}$} &$\in$&$\mathcal{B}_n^2.  $
\end{tabular}}\end{center}
In the second case, by (\ref{deltpois}) and Lemma \ref{0id},
\begin{center}\adjustbox{max width=\linewidth}{\begin{tabular}{lclcl}
$x_{i_2}\{x_{i_3},\{x_{i_1},\{x_{i_4},x_{i_5}\}\}\}$&$=
$&${\delta}^{-1}\{x_{i_3},x_{i_2}\{x_{i_1},\{x_{i_4},x_{i_5}\}\}\}-\{x_{i_3},x_{i_2}\}\{x_{i_1},\{x_{i_4},x_{i_5}\}\}$&$=$\\
\multicolumn{1}{r}{$-\{x_{i_3},x_{i_2}\}\{x_{i_1},\{x_{i_4},x_{i_5}\}\}$}&$=$&$- {\delta}^{-1}\{x_{i_1},\{x_{i_4},x_{i_5}\}\{x_{i_3},x_{i_2}\}\}+\{x_{i_4},x_{i_5}\}\{x_{i_1},\{x_{i_3},x_{i_2}\}\}$&$=$\\
\multicolumn{3}{r}{$-  \delta^{-2}
\{x_{i_1}, \{\{x_{i_4},x_{i_5}\}x_{i_3},x_{i_2}\}\} 
+ {\delta}^{-1}\{x_{i_1}, \{\{x_{i_4},x_{i_5}\},x_{i_2}\}x_{i_3}\} 
+\{x_{i_4},x_{i_5}\}\{x_{i_1},\{x_{i_3},x_{i_2}\}\}$}&$=$\\
\multicolumn{1}{r}{$\{x_{i_4},x_{i_5}\}\{x_{i_1},\{x_{i_3},x_{i_2}\}\}$}&$
=$&\multicolumn{1}{r}{$ {\delta}^{-1}\{x_{i_1}\{x_{i_4},x_{i_5}\},\{x_{i_3},x_{i_2}\}\}-x_{i_1}\{\{x_{i_4},x_{i_5}\},\{x_{i_3},x_{i_2}\}\}$}&$=$\\
\multicolumn{3}{r}{$-x_{i_1}\{\{x_{i_4},x_{i_5}\},\{x_{i_3},x_{i_2}\}\}$} &$\in$&$\mathcal{B}_n^2.$
\end{tabular}}\end{center}

In the third case, by (\ref{deltpois}), second case and Lemma \ref{0id},
\begin{center}\adjustbox{max width=\linewidth}{\begin{tabular}{lclcccc}
$x_{i_2}\{x_{i_3},\{x_{i_4},\{x_{i_1},x_{i_5}\}\}\}$&$=
$&$\{x_{i_1},x_{i_5}\}\{x_{i_4},\{x_{i_3},x_{i_2}\}\}$&$=$\\
\multicolumn{3}{r}{${\delta}^{-1}\{x_{i_1}\{x_{i_4},\{x_{i_3},x_{i_2}\}\},x_{i_5}\}-x_{i_1}\{\{x_{i_4},\{x_{i_3},x_{i_2}\}\},x_{i_5}\}$}&$=$ & {$-x_{i_1}\{\{x_{i_4},\{x_{i_3},x_{i_2}\}\},x_{i_5}\}$}&$\in$&$\mathcal{B}_n^2.$
\end{tabular}}\end{center}

In the general case, it is enough to prove the statement for a monomial
$$x_{i_2}\{x_{i_3},\{\cdots \{x_{i_{1}},x_{i_n}\}\cdots\}\},$$
where $i_1< {\rm min}\{i_2,\ldots,i_n\}$. By (\ref{deltpois}), Lemma \ref{0id} and inductive hypothesis,
\begin{center}\adjustbox{max width=\linewidth}{\begin{tabular}{lclclclc}
$x_{i_2}\{x_{i_3},\{\cdots \{x_{i_{1}},x_{i_n}\}\cdots\}\}$&$=$&$
{\delta}^{-1}\{x_{i_3},x_{i_2}\{\cdots \{x_{i_{1}},x_{i_n}\}\cdots\}\}-\{\cdots \{x_{i_{1}},x_{i_n}\}\cdots\}\{x_{i_3},x_{i_2}\}$&$=$\\
\multicolumn{3}{r}{$-\{x_{i_3},x_{i_2}\}\{\cdots \{x_{i_{1}},x_{i_n}\}\cdots\} \ =\ x_{i_{1}}\{\cdots \{\{x_{i_3},x_{i_2}\},x_{i_n}\}\cdots\}$} & $\in$&$\mathcal{B}_n^2.$
\end{tabular}}\end{center}
So, the second set of monomials corresponds to $\mathcal{B}_n^2$.
By associative and commutative identities of operation $\cdot$, the third set of monomials corresponds to $\mathcal{B}_n^3$.
We proved that the set $\cup_i\mathcal{B}_i$ is a spanning set of algebra $\delta$-${\rm P}(X)$. 

Now, let us prove the linear independence of the set $\cup_i\mathcal{B}_i$.
Let $A(X)$ be an algebra with a basis  $\cup_i\mathcal{B}_i$ and multiplications $\cdot$ and $\{\cdot,\cdot\}$. We define the multiplication on the monomials $\cup_i\mathcal{B}_i$ up to degree $5$ consistently with the identities of Lie algebras, commutative associative algebras, and identity (\ref{deltpois}). Starting from degree $6$, if $A=x_j$ and $B\in\cup_i\mathcal{B}_i$, then $A\cdot B$ and $\{A,B\}$ are defined as follows:
\[\adjustbox{max width=\linewidth}{$\displaystyle \begin{gathered}
\begin{cases}
   x_j\cdot \{x_{i_1},\{x_{i_2},\{\cdots \{x_{i_{n-1}},x_{i_n}\}\cdots\}\}\}\ =\ x_{i_k} \{x_{i_1},\{x_{i_2},\{\cdots \{x_{i_{n-1}},x_{i_n}\}\cdots\}\}\},\text{$i_k$ has a minimal}\\
   \text{index and remained part as in pure Lie algebra,}\\
   x_j\cdot (x_{i_1}\{x_{i_2},\{\cdots \{x_{i_{n-1}},x_{i_n}\}\cdots\}\})\ = \ 0, \\
   x_j\cdot (x_{i_1}x_{i_2}\cdots x_{i_{n-1}}x_{i_n})\ =\ x_{i_1}x_{i_2}\cdots x_j \cdots x_{i_{n-1}}x_{i_n},\; i_1\leq i_2\leq\cdots\leq j\leq\cdots\leq i_n,\\
   \{x_j,\{x_{i_1},\{x_{i_2},\{\cdots \{x_{i_{n-1}},x_{i_n}\}\cdots\}\}\}\}\;\; \text{as in pure Lie algebra}\\
   \{x_j,x_{i_1}\{x_{i_2},\{\cdots \{x_{i_{n-1}},x_{i_n}\}\cdots\}\}\} \ =\ 0, \\
   \{x_j,x_{i_1}x_{i_2}\cdots x_{i_{n-1}}x_{i_n}\}\ = \ 0. \\
\end{cases}
\end{gathered}$}\]
Let $b_i, b^*_i\in \mathcal{B}_{n}^i$ and ${\rm deg} \ (b_i),{\rm deg} \ (b^*_i)>1$. Then 
\[\adjustbox{max width=\linewidth}{$\displaystyle \begin{gathered}
\begin{cases}
b_1\cdot b^*_1\ =\ x_{i_k} \{x_{i_1},\{x_{i_2},\{\cdots \{x_{i_{n-1}},x_{i_n}\}\cdots\}\}\},\;\; \text{$i_k$ has a minimal index and}\\
\text{remained part as in pure Lie algebra,}\\
b_1\cdot b^*_2\ =\ b_1\cdot b^*_3 \ =\ b_2\cdot b^*_2\ =\ b_2\cdot b^*_3\ =\ 0,\\
b_3\cdot b^*_3\;\; \text{as in pure commutative associative algebra},\\
\{b_1,b^*_1\}\;\; \text{as in pure Lie algebra},\\
\{b_1,b^*_2\}\ =\ \{b_1,b^*_3\} \ =\ \{b_2,b^*_2\}\ =\ \{b_2,b^*_3\} \ =\ \{b_3,b^*_3\} \ = \ 0.
\end{cases}
\end{gathered}$}\]

A straightforward calculation shows that the algebra $A(X)$ satisfies the Jacobi identity, the associativity identity, and (\ref{deltpois}). It remains to note that $A(X)\cong\delta$-${\rm P}(X)$, which gives the linear independence of monomials $\cup\mathcal{B}_i$.
\end{proof}

\begin{corollary} For $n\geq 5$, we have
$${\rm dim}\big(\delta\textrm{-}{\rm P}(n)\big)=(n-1)!+(n-2)!+1$$
where $\delta\textrm{-}{\rm P}(n)$ is $n$-th component of the operad $\delta\textrm{-}{\rm P}$.
\end{corollary}

\begin{remark}
The basis of the free $0$-Poisson algebra is the same as the basis of the free Poisson algebra. It is easy to see that the spanning elements of ${0}\textrm{-}{\rm P}(X)$ are
$$\{A_1\}\{A_2\}\cdots\{A_k\},$$
where $\{A_i\}$ is a pure basis monomial of the free Lie algebra. It is not difficult to see that all compositions are trivial, which gives the linear independence of these monomials.
\end{remark}

\subsection{Free mixed-Poisson algebras}\label{mixed}

Denote by $\mathcal{V}_n$ the set of basis monomials of degree $n$ of the free Lie algebra, and by ${\rm MP}(X)$ the free mixed-Poisson algebra generated by a countable set $X$. We set
$$\mathcal{W}_n=\mathcal{V}_n\cup \{x_{i_1} x_{i_2}\cdots x_{i_n}\},$$
where $i_1\leq i_2\leq\ldots\leq i_n$.

\begin{theorem}\label{baseSP}
The set $\cup_i\mathcal{W}_i$ is a basis of the free mixed-Poisson algebra.
\end{theorem}
\begin{proof}
By the definition of mixed-Poisson algebras, we have the identities
$$[a b,c]=0,\;\;\;[a, b] c=0.$$
Hence, $\cup_i\mathcal{W}_i$ is a spanning set of ${\rm MP}(X)$.

Now, let us prove the linear independence of the set $\cup_i\mathcal{W}_i$.
Let $A(X)$ be an algebra with a basis  $\cup_i\mathcal{W}_i$ and operations $\cdot$ and $\{\cdot,\cdot\}$ as follows:

\[\adjustbox{max width=\linewidth}{$\displaystyle \begin{gathered}
\begin{cases}
    \{L_i,L_j\}\;\; \textit{is defined as in free Lie algebra,}\\
    C_i\cdot C_j\;\; \textit{is defined as in free commutative associative algebra,}\\
    C_i\cdot L_i=L_i\cdot C_i=\{C_i,L_i\}=\{L_i,C_i\}=0,
\end{cases}
\end{gathered}$}\]
where $L_i,L_j\in\bigcup_i\mathcal{V}_i$ and $C_i,C_j$ are monomials containing only the operation $\cdot$. It is easy to verify that the algebra $A(X)$ satisfies all identities of mixed-Poisson algebras. It remains to note that $A(X)\cong {\rm MP}(X)$, which gives the linear independence of monomials $\cup\mathcal{W}_i$.
\end{proof}

\begin{corollary} For arbitrary $n$, we have
$${\rm dim}({\rm MP}(n))=(n-1)!+1$$
where ${\rm MP}(n)$ is $n$-th component of the operad ${\rm MP}$.
\end{corollary}

 \section{Operads} \label{dualoper}

 \subsection{Dual operad of the $\delta$-Poisson operad} 

In this section, following \cite{GK94}, we compute the Koszul dual operad $\delta$-${\rm P}^!$, where the operad $\delta$-${\rm P}$ is governed by the variety of $\delta$-Poisson algebras. We denote by $\mu$ and $\nu$ operations of $\delta$-Poisson algebra, namely, $\mu=\{x_1,x_2\}$ and $\nu=x_1\cdot x_2$. Then $\delta$-${\rm P}(2)= {\rm span}\{\mu,\nu\}$. Any multilinear term of degree $3$ can be written as a composition 
of $\delta$-${\rm P}(2)$ and action from ${\mathbb S}_3$. For example,
\begin{center}\adjustbox{max width=\linewidth}{\begin{tabular}{rclrcl}
$(x_1\cdot x_2)\cdot x_3 $&$=$&$\nu\otimes \nu,$&  $x_1\cdot (x_2\cdot x_3)$&$=$&$(13)\nu^{(12)}\otimes \nu^{(12)},$\\
$\{\{x_1,x_2\},x_3\} $&$=$&$\mu\otimes \mu,$&$ \{x_1,\{x_2,x_3\}\}$&$=$&$(13)\mu^{(12)}\otimes \mu^{(12)},$\\
$\{x_1\cdot x_2,x_3\}$&$=$&$\mu\otimes \nu,$&$ x_1\cdot \{x_2,x_3\}$&$=$&$(13)\mu^{(12)}\otimes \nu^{(12)},$\\
$\{x_1,x_2\}\cdot x_3$&$=$&$\mu\otimes \nu,$&$ \{x_1,x_2\cdot x_3\}$&$=$&$(13)\mu^{(12)}\otimes \nu^{(12)}.$
\end{tabular}}\end{center}
\noindent The remaining terms can be represented by the given elements under actions $(12)$ and $(23)$.

Let us order the basis monomials of $\delta$-${\rm P}(3)$ that contain exactly one operation $\mu$ and one operation $\nu$ as follows:
$$\{xy,z\},\{xz,y\},\{yz,x\},\{x,y\}z,\{x,z\}y,\{y,z\}x.$$
Relative to the defined order, the vectors representing the relation $\{xy,z\}-\delta\{x,z\}y-\delta\{y,z\}x$ of $\delta$-${\rm P}$
may be presented by the rows of the following matrix:

\[\adjustbox{max width=\linewidth}{$\displaystyle \begin{matrix}
1 &0 &0 & 0 &-\delta & -\delta \\
0 &1 &0 & -\delta &0 & \delta \\
0 &0 &1 & \delta &\delta & 0 \\
\end{matrix}$}\]
Since all relations of the operad $\delta$-${\rm P}$ are homogeneous, it is enough to compute the sign-twisted orthogonal complement $U^\perp \subset O_3^\vee $  
for the ${\mathbb S}_3$-module $U$ spanned by these vectors \cite{GK94}.
Finally, we obtain the defining relations of the Koszul dual operad $\delta$-${\rm P}^!$.
It turns out that a $\delta$-${\rm P}^!$-algebra has two operations $\mu^\vee$ and $\nu^\vee$, where the first one is commutative associative and the second one is a Lie multiplication, with the additional identities
\begin{center}\adjustbox{max width=\linewidth}{\begin{tabular}{rcl}
$-\delta\{x,y\}z+\delta\{y,z\}x+\{xz,y\} $&$=$&$0,$\\
$\delta\{x,z\}y+\delta\{y,z\}x+\{xz,y\}+\{yz,x\} $&$=$&$0,$\\
$\{xy,z\}+\{xz,y\}+\{yz,x\} $&$=$&$0.$\\
\end{tabular}}\end{center}
These identities are equivalent to the identity
$$\{xy,z\}-\delta\{x,z\}y-\delta\{y,z\}x=0$$
which gives the following result:
\begin{theorem}\label{26}
For any $\delta$, the operad $\delta$-${\rm P}$ is self-dual.
\end{theorem}

\begin{remark}
An analogous result holds for the transposed $\delta$-Poisson operad, i.e., for any $\delta$, the transposed $\delta$-Poisson operad is self-dual.
\end{remark}

It is well-known that the  Poisson operad is Koszul. However, for the anti-Poisson operad, we obtain the opposite result:
\begin{theorem}\label{27}
The operad $\AP$ is not Koszul, where the operad $\AP$ is governed by the variety of anti-Poisson algebras.
\end{theorem}
\begin{proof}
Calculating the dimensions of the components of the operad $\AP$ by means of the software \cite{DotsHij}, 
we obtain the following table:
\begin{center}
\begin{tabular}{c|ccccccc}
 $n$ & 1 & 2 & 3 & 4 & 5 & 6 & 7 \\
 \hline 
 $\dim(\AP(n)) $ & 1 & 2 & 6 & 12 & 31 & 145 & 841
\end{tabular}
\end{center}
According to this table, the beginning of the Hilbert series of the operads $\AP$ and $\AP^{!}$ is
$$H(t)=H^!(t)=-t+t^2-t^3+t^4/2-31t^5/120+O(t^6).$$
Thus,
$$H(H^!(t))=t + 91t^5/60+O(t^6)\neq t.$$
By \cite{GK94}, the operad $\AP$ is not Koszul.
\end{proof}

\begin{remark}
The statement of Theorem \ref{27} also holds for $\delta\textrm{-}{\rm P}$ with $\delta\notin\big\{ 0,1\big\}$, since the dimensions of the components of these operads are the same.
\end{remark}

\subsection{Mixed-Poisson operad}

\begin{theorem}
The mixed-Poisson operad is Koszul.
\end{theorem}
\begin{proof}
By Theorem \ref{baseSP}, a Gr\"obner basis of the operad ${\rm MP}$ consists of the following rewriting rules:
\begin{center}\adjustbox{max width=\linewidth}{\begin{tabular}{rclrcl}
$(a\circ b)\circ c$&$\rightarrow $&$a\circ(b\circ c),$ &
$\{a,b\}\circ c$&$\rightarrow$&$ 0,$\\
$\{\{a,b\},c\} $&$\rightarrow $&$\{\{a,c\},b\}+\{a,\{b,c\}\},$ &
$\{a,b\circ c\} $&$\rightarrow$&$ 0.$
\end{tabular}}\end{center}
By \cite{DotsHij}, this is a quadratic Gr\"obner basis; hence, the operad is Koszul \cite{Bremner-Dotsenko}.
\end{proof}

\begin{theorem}\label{36}
The dual mixed-Poisson operad is isomorphic to the free product of operads Lie and As-Com, i.e., we have
$${\rm MP}^!\cong\Lie\ast\As\textrm{-}\Com.$$
\end{theorem}
\begin{proof}
First, let us find the defining identities of the operad ${\rm MP}^!$. We compute it in the same way as the one given in the previous section.

Let us order the basis monomials of ${\rm MP}(3)$ as follows:
$$\{xy,z\}, \ \{xz,y\},\ \{yz,x\},\ \{x,y\}z, \ \{x,z\}y,\ \{y,z\}x.$$
Relative to the defined order, the vectors representing the relation $\{xy,z\}+\{x,z\}y-\{y,z\}x$ of $\MP$
may be presented by the rows of the following matrix:
\[\adjustbox{max width=\linewidth}{$\displaystyle \begin{matrix}
1 &0 &0 & 0 & 1 & -1 \\
1 &0 &0 & 0 & -1 & 1 \\
0 &1 &0 & -1 &0 & -1 \\
0 &1 &0 & 1 &0 & 1 \\
0 &0 &1 & 1 &-1 & 0 \\
0 &0 &1 & -1 &1 & 0 \\
\end{matrix}$}\]
Since all relations of the operad ${\rm MP}$ are homogeneous, it is enough to compute the sign-twisted orthogonal complement $U^\perp \subset O_3^\vee $  
for the ${\mathbb S}_3$-module $U$ spanned by these vectors \cite{GK94}.
As a result, there are no identities containing both operations $\{\cdot,\cdot\}$ and $\cdot$.
This proves the result.
\end{proof}

For more details on free products of operads, see \cite{SarAbd}, where a method for calculating the dimensions of the components of the operad $\Var_1\ast \Var_2$ is given. For $\Var_1=\Lie$ and $\Var_2=\As\textrm{-}\Com$, we obtain
\begin{center}
\begin{tabular}{c|ccccccc}
 $n$ & 1 & 2 & 3 & 4 & 5 & 6 & 7\\
 \hline 
 $\dim({\rm MP}^!(n)) $ & 1 & 2 & 9 & 67 & 695 & 9256 & 150570
\end{tabular}

\end{center}



\appendix
\section*{Appendix A. Differences from the published version}
\addcontentsline{toc}{section}{Appendix A. Differences from the published version}

\noindent\textbf{Proposition \ref{deltaxyzt}:} part 2 has been corrected (an additional identity is needed).

\noindent\textbf{Proposition \ref{idtdpa}:} the assumption on $\delta$ has been corrected (identity \eqref{idtp2} fails for $\delta=0$).

\noindent\textbf{Theorem \ref{simpletrdelta}:} the proof has been corrected.

\noindent\textbf{Proposition \ref{unitalP}:} the assumption $\delta\neq1$ has been added in part 2.

\end{document}